\documentclass{amsart}
\usepackage{graphics}
\usepackage{amsfonts,amsmath}
\begin{document}

 \newtheorem{thm}{Theorem}[section]
 \newtheorem{cor}[thm]{Corollary}
 \newtheorem{lem}[thm]{Lemma}{\rm}
 \newtheorem{prop}[thm]{Proposition}

 \newtheorem{defn}[thm]{Definition}{\rm}
 \newtheorem{assumption}[thm]{Assumption}
 \newtheorem{rem}[thm]{Remark}
 \newtheorem{ex}{Example}\numberwithin{equation}{section}

\def\br{\bar{\rho}}
\def\la{\langle}
\def\C{\mathcal{C}}
\def\ra{\rangle}
\def\e{{\rm e}}
\def\x{\mathbf{x}}
\def\by{\mathbf{y}}
\def\bz{\mathbf{z}}
\def\F{\mathcal{F}}
\def\R{\mathbb{R}}
\def\T{\mathbf{T}}
\def\N{\mathbb{N}}
\def\K{\mathbf{K}}
\def\bK{\overline{\mathbf{K}}}
\def\Q{\mathbf{Q}}
\def\M{\mathbf{M}}
\def\O{\mathbf{O}}
\def\C{\mathbf{C}}
\def\P{\mathbf{P}}
\def\Z{\mathbf{Z}}
\def\H{\mathcal{H}}
\def\A{\mathbf{A}}
\def\V{\mathbf{V}}
\def\AA{\overline{\mathbf{A}}}
\def\B{\mathbf{B}}
\def\c{\mathbf{c}}
\def\L{\mathbf{L}}
\def\bS{\mathbf{S}}
\def\H{\mathcal{H}}
\def\I{\mathbf{I}}
\def\Y{\mathbf{Y}}
\def\X{\mathbf{X}}
\def\G{\mathbf{G}}
\def\f{\mathbf{f}}
\def\z{\mathbf{z}}
\def\y{\mathbf{y}}
\def\d{\hat{d}}
\def\bx{\mathbf{x}}
\def\y{\mathbf{y}}
\def\v{\mathbf{v}}
\def\g{\mathbf{g}}
\def\w{\mathbf{w}}
\def\b{\mathcal{B}}
\def\a{\mathbf{a}}
\def\b{\mathbf{b}}
\def\q{\mathbf{q}}
\def\u{\mathbf{u}}
\def\h{\mathbf{h}}
\def\s{\mathcal{S}}
\def\bs{\mathbf{s}}
\def\br{\mathbf{r}}
\def\cc{\mathcal{C}}
\def\co{{\rm co}\,}
\def\cop{{\rm COP}\,}
\def\tg{\tilde{f}}
\def\tx{\tilde{\x}}
\def\supmu{{\rm supp}\,\mu}
\def\supnu{{\rm supp}\,\nu}
\def\m{\mathcal{M}}
\def\s{\mathcal{S}}
\def\k{\mathcal{K}}
\def\la{\langle}
\def\ra{\rangle}
\def\psd{{\rm Psd}}

\title[nonnegativity]{Inverse polynomial optimization}
\author{Jean B. Lasserre}
\address{LAAS-CNRS and Institute of Mathematics\\
University of Toulouse\\
LAAS, 7 avenue du Colonel Roche\\
31077 Toulouse C\'edex 4,France}
\email{lasserre@laas.fr}
\date{}

\begin{abstract}
We consider the inverse optimization problem associated with
the polynomial program
$f^*=\min \{f(\x):\x\in\K\}$ and a given current feasible solution $\y\in\K$. 
We provide a systematic numerical scheme to compute an inverse
optimal solution. That is, we compute
a polynomial $\tilde{f}$ (which may be of same degree as $f$ if desired) 
with the following properties: 
(a) $\y$ is a global minimizer
of $\tilde{f}$ on $\K$ with a Putinar's certificate with 
an {\it a priori} degree bound $d$ fixed,
and (b), $\tilde{f}$ minimizes $\Vert f-\tilde{f}\Vert$ (which can be
the $\ell_1$, $\ell_2$ or $\ell_\infty$-norm of the coefficients)
over all polynomials with such properties. Computing $\tilde{f}_d$
reduces to solving a semidefinite program whose optimal value
also provides a bound on how far 
is $f(\y)$ from the unknown optimal value $f^*$. 
The size of the semidefinite program can be adapted to 
computational capabilities available.
Moreover, if one uses the $\ell_1$-norm,
then $\tilde{f}$ takes a simple and explicit {\it canonical} form.
Some variations are also discussed.
\end{abstract}

\keywords{Inverse optimization; positivity certificate; mathematical programming; global optimization; 
semidefinite programming}

\subjclass{90C26 47J07 65J22 49K30 90C22}

\maketitle

\section{Introduction}

Let $\P$ be the optimization problem $f^*= \min \:\{ f(\x)\::\:\x\in\K\:\}$,
where 
\begin{equation}
\label{setk}
\K:=\{\x\in\R^n\,:\, g_j(\x)\geq0,\: j=1,\ldots,m\},
\end{equation}
for some polynomials $f,(g_j)\subset\R[\x]$. This framework is rather general as it encompasses
a large class of important optimization problems, including non convex and 
discrete optimization problems.

Problem $\P$ is in general NP-hard and one is often satisfied with a local minimum
which can be obtained by running some local minimization algorithm among those available in the literature.
Typically in such algorithms, at a current iterate (i.e. some feasible solution $\y\in\K$), one checks whether
some optimality conditions (e.g. the Karush-Kuhn-Tucker (KKT) conditions) are satisfied within some
$\epsilon$-tolerance. However, as already mentioned those conditions are only valid
for a local minimum, and in fact, even only for a stationary point of the Lagrangian.
Moreover, in many practical situations
the criterion $f$ to minimize is subject to modeling errors
or is questionable. In such a situation, the practical meaning of a
local (or global) minimum $f^*$ (and local (or global) minimizer) also becomes 
questionable. It could well be that the current solution $\y$ is in fact a global
minimizer of an optimization problem $\P'$ with same feasible set as $\P$
but with a different criterion $\tilde{f}$. Therefore, if $\tilde{f}$
is close enough to $f$, one might not be willing to spend 
an enormous computing time and effort to find the global (or even local) minimum $f^*$
because one might be already satisfied with the current iterate $\y$ as a global minimizer of $\P'$.

{\bf Inverse Optimization} is precisely concerned with the above issue of determining
a criterion $\tilde{f}$ as close to $f$ as possible, and for which the current solution $\y$ is
an optimal solution of $\P'$ with this new criterion $\tilde{f}$.
Pioneering work in Control dates back to Freeman and Kokotovic
\cite{koko} for optimal stabilization. Whereas it was known that
every value function of an optimal stabilization problem is also a Lyapunov function for the closed-loop system,
in \cite{koko} the authors show the converse, that is, every Lyapunov
function for every stable closed-loop system is also a {\it value function} for a meaningful optimal stabilization problem. In optimization, pioneering works in this direction date back to Burton and Toint and \cite{burton} for shortest path problems, 
and Zhang and Liu \cite{zhang1,zhang2}, Huang and Liu \cite{huang}, and
Ahuja and Orlin and \cite{ahuja} for linear programs in the form
$\min \{\c'\x\,:\,\A\x\geq \b;\:\br \leq\x\leq \bs\}$ (and with the $\ell_1$-norm). For the latter, the inverse problem is again a linear program of the same form. Similar results also hold
for inverse linear programs with the $\ell_\infty$-norm as shown in Ahuja and Orlin \cite{ahuja} while
Zhang et al. \cite{zhangetal} provide a column generation method for the inverse shortest path problem.
In Heuberger \cite{heuberger} the interested reader
will find a nice survey on inverse optimization for linear programming and combinatorial optimization problems.
For integer programming, Schaefer \cite{schaefer} characterizes the feasible set of
cost vectors $c\in\R^n$ that are candidates for inverse optimality. It is the projection on $\R^n$ of 
a (lifted) convex polytope obtained from the super-additive dual of integer programs.
Unfortunately and as expected, the dimension of
of the lifted polyhedron (before projection) is exponential in the input size of the problem.
Finally, for linear programs Ahmed and Guan \cite{ahmed} have considered 
the variant called {\it inverse optimal value} problem 
in which one is interested in finding a linear criterion $c\in C\subset\R^n$ for which
the optimal value is the closest to a desired specified value. Perhaps surprisingly,
they proved that such a problem is NP-hard.\\

As the reader may immediately guess, in inverse optimization the main difficulty 
lies in having a tractable characterization of global optimality 
for a given current point $\y\in\K$ and some candidate criterion $\tilde{f}$. This is why 
most of all the above cited works address linear programs or 
combinatorial optimization problems for which some
characterization of global optimality is available and can be (sometimes) effectively
used for practical computation. For instance,
the characterization of global optimality for integer programs described in Schaefer \cite{schaefer} is
via the superadditive dual  of Wolsey \cite[\S 2]{wolsey} which is exponential in the 
problem size, and so prevents from its use in practice. 

This perhaps explains why inverse (non linear) optimization has not attracted much attention in the past, and
it  is a pity since inverse optimality could provide an alternative stopping criterion 
at a feasible solution $\y$ obtained by a (local) optimization algorithm.

The novelty of the present paper is to provide
a systematic numerical scheme for computing an inverse optimal solution associated with
the polynomial program $\P$ and a given feasible solution $\y\in\K$.
It consists of solving a semidefinite program\footnote{A semidefinite program is a convex (conic)
optimization problem that can be solved efficiently. For instance, up to arbitrary (fixed) precision and using some interior point algorithms,
it can be solved in time polynomial in the input size of the problem. For more details the interested reader is referred to
e.g. Wolkowicz et al. \cite{handbook} and the many references therein.}
 whose size can be adapted to
the problem on hand, and so  is tractable (at least for moderate size problems
and possibly for larger size problems if sparsity is taken into account).
Moreover, if one uses the $\ell_1$-norm
then the inverse-optimal objective function exhibits a simple and remarkable {\it canonical} (and sparse) form.

{\bf Contribution.} In this paper we investigate the inverse optimization problem for
polynomial optimization problems $\P$ as in (\ref{setk}), i.e., in a rather general context which
includes nonlinear and nonconvex optimization problems and  in particular,
0/1 and mixed integer nonlinear programs. Fortunately, in such a context,
Putinar's Positivstellensatz \cite{putinar} provides us with a very
powerful certificate of global optimality that can be adapted to
the actual computational capabilities for a given problem size.
More precisely, and assuming $\y=0$ (possibly after a change of variable $\x'=\x-\y$),
in the methodology that we propose,
one computes the coefficients of a polynomial
$\tilde{f}_d\in\R[\x]$ of same degree $d_0$ as $f$ (or possibly larger degree if desired and/or possibly with some
additional constraints), such that:
\begin{itemize}
\item 
$0$ is a global minimizer
of the related problem $\min_\x\{ \tilde{f}_d(\x)\,:\,\x\in\K\}$, with a Putinar's certificate of optimality with degree bound $d$ (to be explained later). 
\item $\tilde{f}_d$ minimizes $\Vert \tilde{f}-f\Vert_k$ (where depending on $k$, 
$\Vert \cdot\Vert_k$ is the $\ell_1$, $\ell_2$ or $\ell_\infty$-norm of the coefficient vector) over all polynomals $\tilde{f}$ of degree $d_0$, having the previous property.
\end{itemize}

Assuming $\K$ is compact (hence $\K\subseteq [-1,1]^n$ possibly after a change of variable),
it turns out that the optimal value $\rho_d:=\Vert \tilde{f}_d-f\Vert_k$ also measures how close is $f(0)$ to the global optimum
$f^*$ of $\P$, as we also obtain that $f^*\leq f(0)\leq f^*+\rho_d$ if $k=1$ and similarly
$f^*\leq f(0)\leq f^*+\rho_d\,{n+d_0\choose n}$ if $k=\infty$. 

In addition, for the $\ell_1$-norm we prove that 
$\tilde{f}_d$ has a simple {\it canonical form}, namely 
\[\tilde{f}_d\,=\,f+\b'\x+\sum_{i=1}^n\lambda_i\,x_i^2,\]
for some vector $\b\in\R^n$, and nonnegative vector $\lambda\in\R^n$, optimal solutions of a semidefinite program.
(For 0/1 problems it further simplifies to $\tilde{f}_d=f+\b'\x$ only.)
This canonical form is sparse as $\tilde{f}_d$ differs from $f$ in at most $2n$ entries only ($\ll {n+d_0\choose n}$). It illustrates
the sparsity properties of optimal solutions of $\ell_1$-norm minimization problems, already observed in other contexts
(e.g., in some compressed sensing applications).

Importantly, to compute $\tilde{f}_d$, one has to solve a semidefinite program of size parametrized by $d$,
where $d$ is chosen so that the size of semidefinite program associated with Putinar's certificate 
(with degree bound $d$) is compatible with current semidefinite solvers available. (Of course,
even if $d$ is relatively small, one is still restricted to problems of relatively modest size.)
Moreover, when $\K$ is compact, generically 
$\tilde{f}_d$ is an optimal solution of the ``ideal inverse optimization problem" 
provided that $d$ is sufficiently large!

In addition, one may also consider several additional options: 

$\bullet$ Instead of looking for a polynomial $\tilde{f}$ of same degree as $f$,
one might allow polynomials of higher degree, and/or restrict 
certain coefficients of $\tilde{f}$ to be the same as those of $f$ (e.g. for 
structural modeling reasons).

$\bullet$ One may restrict $\tilde{f}$ to a certain class of functions, e.g.,
quadratic polynomials and even convex quadratic polynomials.
In the latter important case and if the $g_j$'s that define $\K$ are concave,
the procedure to compute an optimal solution
$\tilde{f}(\x)=\tilde{\b}'\x+\x'\tilde{\Q}\x$ simplifies and reduces to solving separately a linear program 
(for computing $\tilde{\b}$) and a semidefinite program (for computing $\tilde{\Q}$).

The paper is organized as follows. In a first introductory section we present the notation, definitions, and
the ideal inverse optimization problem. We then describe how a practical inverse optimization problem
reduces to solving a semidefinite program and exhibit the canonical form of the optimal solution for the $\ell_1$-norm.
We also provide additional results, e.g., an asymptotic analysis when the degree bound in Putinar's certificate 
increases and also the particular case where one searches for a convex candidate criterion.

\section{Notation, definitions and preliminaries}

\subsection{Notation and definitions}

Let $\R[\x]$ (resp. $\R[\x]_d$) denote the ring of real polynomials in the variables
$\x=(x_1,\ldots,x_n)$ (resp. polynomials of degree at most $d$), whereas $\Sigma[\x]$ (resp. $\Sigma[\x]_d$) denotes 
its subset of sums of squares (s.o.s.) polynomials (resp. of s.o.s. of degree at most $2d$).
For every
$\alpha\in\N^n$ the notation $\x^\alpha$ stands for the monomial $x_1^{\alpha_1}\cdots x_n^{\alpha_n}$ and for every $i\in\N$, let $\N^{p}_d:=\{\beta\in\N^n:\sum_j\beta_j\leq d\}$ whose cardinal is $s(d)={n+d\choose n}$.
A polynomial $f\in\R[\x]$ is written 
\[\x\mapsto f(\x)\,=\,\sum_{\alpha\in\N^n}\,f_\alpha\,\x^\alpha,\]
and $f$ can be identified with its vector of coefficients $\f=(f_\alpha)$ in the canonical basis $(\x^\alpha)$, $\alpha\in\N^n$. Denote by $\mathcal{S}^t\subset\R^{t\times t}$ the space of real symmetric matrices,
and for any $\A\in\mathcal{S}^t$ the notation $\A\succeq0$ stands for $\A$ is positive semidefinite.
For $f\in\R[\x]_d$, let
\[\Vert f\Vert_k\,=\,\left\{\begin{array}{ll}
\displaystyle\sum_{\alpha\in\N^n_{d}}\vert f_\alpha\vert&\mbox{if $k=1$,}\\
\displaystyle\sum_{\alpha\in\N^n_{d}} f_\alpha^2&\mbox{ if $k=2$,}\\
\displaystyle\max\,\{\vert f_\alpha\vert\,:\:\alpha\in\N^n_{d}\}&\mbox{ if $k=\infty$.}\end{array}\right.\]
A real sequence $\z=(z_\alpha)$, $\alpha\in\N^n$, has a {\it representing measure} if
there exists some finite Borel measure $\mu$ on $\R^n$ such that 
\[z_\alpha\,=\,\int_{\R^n}\x^\alpha\,d\mu(\x),\qquad\forall\,\alpha\in\N^n.\]

Given a real sequence $\z=(z_\alpha)$ define the linear functional $L_\z:\R[\x]\to\R$ by:
\[f\:(=\sum_\alpha f_\alpha\x^\alpha)\quad\mapsto L_\z(f)\,=\,\sum_{\alpha}f_\alpha\,z_\alpha,\qquad f\in\R[\x].\]
\subsection*{Moment matrix}
The {\it moment} matrix associated with a sequence
$\z=(z_\alpha)$, $\alpha\in\N^n$, is the real symmetric matrix $\M_d(\z)$ with rows and columns indexed by $\N^n_d$, and whose entry $(\alpha,\beta)$ is just $z_{\alpha+\beta}$, for every $\alpha,\beta\in\N^n_d$. 
Alternatively, let
$\v_d(\x)\in\R^{s(d)}$ be the vector $(\x^\alpha)$, $\alpha\in\N^n_d$, and
define the matrices $(\B_\alpha)\subset\s^{s(d)}$ by
\begin{equation}
\label{balpha}
\v_d(\x)\,\v_d(\x)^T\,=\,\sum_{\alpha\in\N^n_{2d}}\B_\alpha\,\x^\alpha,\qquad\forall\x\in\R^n.\end{equation}
Then $\M_d(\z)=\sum_{\alpha\in\N^n_{2d}}z_\alpha\,\B_\alpha$.

If $\z$ has a representing measure $\mu$ then
$\M_d(\z)\succeq0$ because
\[\langle\f,\M_d(\z)\f\rangle\,=\,\int f^2\,d\mu\,\geq0,\qquad\forall \,\f\,\in\R^{s(d)}.\]

\subsection*{Localizing matrix}
With $\z$ as above and $g\in\R[\x]$ (with $g(\x)=\sum_\gamma g_\gamma\x^\gamma$), the {\it localizing} matrix associated with $\z$ 
and $g$ is the real symmetric matrix $\M_d(g\,\z)$ with rows and columns indexed by $\N^n_d$, and whose entry $(\alpha,\beta)$ is just $\sum_{\gamma}g_\gamma z_{\alpha+\beta+\gamma}$, for every $\alpha,\beta\in\N^n_d$.
Alternatively, let $\C_\alpha\in\s^{s(d)}$ be defined by:
\begin{equation}
\label{calpha}
g(\x)\,\v_d(\x)\,\v_d(\x)^T\,=\,\sum_{\alpha\in\N^n_{2d+{\rm deg}\,g}}\C_\alpha\,\x^\alpha,\qquad\forall\x\in\R^n.\end{equation}
Then $\M_d(g\,\z)=\sum_{\alpha\in\N^n_{2d+{\rm deg}g}}z_\alpha\,\C_\alpha$.

If $\z$ has a representing measure $\mu$ whose support is 
contained in the set $\{\x\,:\,g(\x)\geq0\}$ then
$\M_d(g\,\z)\succeq0$ because
\[\langle\f,\M_d(g\,\z)\f\rangle\,=\,\int f^2\,g\,d\mu\,\geq0,\qquad\forall \,\f\,\in\R^{s(d)}.\]

With $\K$ as in (\ref{setk}), let $g_0\in\R[\x]$ be the constant polynomial $\x\mapsto g_0(\x)=1$, and 
for every $j=0,1,\ldots,m$, let $v_j:=\lceil ({\rm deg}\,g_j)/2\rceil$.
\begin{defn}
With $d,k\in\N$ and $\K$ as in (\ref{setk}),  let $Q_k(g)\subset\R[\x]$ and $Q^d_k\subset\R[\x]_d$ be the convex cones:
\begin{eqnarray}
\label{put-suff?}
Q(g)&:=&\left\{\:\sum_{k=0}^m \sigma_j\,g_j\::\: \sigma_j\in\Sigma[\x]\quad j=1,\ldots,m\:\right\}.\\
\label{put-suff}
Q_k(g)&:=&\left\{\:\sum_{k=0}^m \sigma_j\,g_j\::\: \sigma_j\in\Sigma[\x]_{k-v_j},\quad j=1,\ldots,m\:\right\}.\\
\label{puT-suff1}
Q_k^d(g)&:=&Q_k(g)\,\cap\,\R[\x]_d
\end{eqnarray}
We say that every element $h\in Q_k(g)$ has a {\it Putinar's certificate} of nonnegativity on $\K$, with degree  bound $k$.
\end{defn}
The cone $Q(g)$ is called the quadratic module associated with the $g_j$'s.
Obviously, if $h\in Q(g)$ the associated s.o.s. polynomials $\sigma_j$'s provide a certificate of nonnegativity of $h$ on $\K$.
The cone $Q(g)$ is said to be {\it Archimedean} if and only if
\begin{equation}
\label{archimedean}
\x\mapsto M-\Vert\x\Vert^2\,\in\,Q(g)\quad\mbox{for some $M>0$.}
\end{equation}
Let $\psd_d(\K)\subset\R[\x]_d$ be the convex cone of polynomials of degree at most $d$, nonnegative on $\K$.
The name ``Putinar's certificate" is coming from the following Putinar's Positivstellensatz.
\begin{thm}[Putinar's Positivstellensatz \cite{putinar}]
\label{th-put}
Let $\K$ be as in (\ref{setk}) and assume that $Q(g)$ is Archimedean.
Then every polynomial $f\in\R[\x]$ strictly positive on $\K$ belongs to $Q(g)$. In addition,
\begin{equation}
\label{limit}
{\rm cl}\,\left(\bigcup_{k=0}^\infty Q^d_k(g)\right)\,=\,\psd_d(\K),\qquad \forall d\in\N.
\end{equation}
\end{thm}
The first statement is just Putinar's Positivstellensatz \cite{putinar} whereas the second statement
is an easy consequence. Indeed let $f\in\psd_d(\K)$. If $f>0$ on $\K$ then 
$f\in Q^d_k(g)$ for some $k$. If $f(\x)=0$ for some $\x\in\K$, let $f_n:=f+1/n$, so that
$f_n>0$ on $\K$ for every $n\in\N$. But then $f_n\in\cup_{k=0}^\infty Q^d_k(g)$ and 
the result follows because $\Vert f_n-f\Vert_1\to0$ as $n\to\infty$.

In fact, by results from Marshall \cite{marshall2} and more recently Nie \cite{nie},
membership in $Q(g)$ is also {\it generic} for polynomials that are only nonnegative on $\K$.
And so Putinar's Positivstellensatz is particularly useful to {\it certify} and enforce that 
a polynomial is nonnegative on $\K$, and in particular the polynomial
 $\x\mapsto f(\x)-f(\y)$ for the inverse optimization problem associated with a feasible solution $\y\in\K$.
 
 Notice that one may also be less demanding and ask $\y$ to be only a global $\epsilon$-minimizer
 for some fixed $\epsilon>0$.  Again Putinar's Positivstellensatz is exactly what we need
 to {\it certify} (global) $\epsilon$-optimality by requiring $f(\cdot)-f(\y)+\epsilon\in Q_k(g)$.
 
\subsection{The ideal inverse problem}

Let $\P$ be the global optimization problem
$f^*=\min_\x\{f(\x)\,:\,\x\in\K\}$ with $\K\subset\R^n$ as in (\ref{setk}),
and $f\in\R[\x]_{d_0}$ where $d_0:={\rm deg}\,f$.

Identifying a polynomial $f\in\R[\x]_{d_0}$ with its vector of coefficients
$\f=(f_\alpha)\in\R^{s(d_0)}$, one may and will identify $\R[\x]_{d_0}$ with the vector space 
$\R^{s(d_0)}$, i.e., $\R[\x]_{d_0}\ni f\leftrightarrow \f\in\R^{s(d_0)}$. Similarly,
the convex cone $\psd_{d_0}(\K)\subset\R[\x]_{d_0}$ can be identified with the convex cone 
$\{\h\in\R^{s(d_0)}: \h\leftrightarrow h\in\psd_{d_0}(\K)\}$ of $\R^{s(d_0)}$. So in the sequel, and unless if necessary, we will not distinguish between $f$ and $\f$.

Next, let $\y\in\K$ and $k\in\{1,2,\infty\}$ both fixed, and consider the following optimization problem $\mathcal{P}$
\begin{equation}
\label{inv-0}
\mathcal{P}:\quad\rho^k=
\displaystyle\min_{\tilde{f}\in\R[\x]_{d_0}}\:\{\,\Vert \tilde{f}-f\Vert_k\::\:\x\mapsto \tilde{f}(\x)-\tilde{f}(\y)\,\in\,\psd_{d_0}(\K)\:\}.
\end{equation}

\begin{thm}
\label{th-ideal}
Let $\K\subset\R^n$ be with nonempty interior.
Problem (\ref{inv-0}) has an optimal solution $\tilde{f}^*\in\R[\x]_{d_0}$. In addition,
$\rho^k=0$ if and only if $\y$ is an optimal solution of $\P$.
\end{thm}
\begin{proof}
Obviously the constant polynomial $\x\mapsto \tilde{f}(\x):=1$ is a feasible solution with associated 
value $\delta:=\Vert \tilde{f}-f\Vert_k$.
Moreover the optimal value of (\ref{inv-0}) is bounded below by $0$. 
Observe that $\Vert \cdot\Vert_k$ defines a norm on $\R[\x]_{d_0}$.
Consider a minimizing sequence $(\tilde{f}^j)\subset\R[\x]_{d_0}$, $j\in\N$, hence such that
$\Vert \tilde{f}^j-f\Vert_k\to\rho^k$ as $j\to\infty$. As we have 
$\Vert \tilde{f}^j-f\Vert_k\leq \delta$ for  every $j$, the sequence $(\tilde{f}^j)$ belongs to the $\ell_k$-ball 
$\B_k(f):=\{\tilde{f}\in\R[\x]_{d_0}\,:\,\Vert \tilde{f}-f\Vert_k\leq\delta\}$, a compact set. Therefore, there is an element
$\tilde{f}^*\in\B_k(f)$ and a subsequence $(j_t)$, $t\in\N$, such that
$\tilde{f}^{j_t}\to\tilde{f}^*$ as $t\to\infty$. Let $\x\in\K$ be fixed arbitrary.
Obviously $(0\leq)\, \tilde{f}^{j_t}(\x)-\tilde{f}^{j_t}(\y)
\to\tilde{f}^*(\x)-\tilde{f}^*(\y)$ as $t\to\infty$, which implies $\tilde{f}^*(\x)-\tilde{f}^*(\y)\geq0$,
and so, as $\x\in\K$ was arbitrary, $\tilde{f}^*-\tilde{f}^*(\y)\geq0$ on $\K$, i.e., $\tilde{f}^*-\tilde{f}^*(\y)\in\psd_{d_0}(\K)$. Finally, we also obtain the desired result 
\[\rho^k\,=\,\lim_{j\to\infty}\Vert \tilde{f}^{j}-f\Vert_k\,=\,\lim_{t\to\infty}\Vert \tilde{f}^{j_t}-f\Vert_k\,=\,\Vert \tilde{f}^*-f\Vert_k.\]
Next, if $\y$ is an optimal solution of $\P$ then $\tilde{f}:=f$ is an optimal solution of 
$\mathcal{P}$ with value $\rho^k=0$. Conversely, if $\rho^k=0$ then $\tilde{f}^*=f$, and so by feasibility of $\tilde{f}^*\,(=f)$ for (\ref{inv-0}),
$f(\x)\geq f(\y)$ for all $\x\in\K$, which shows that $\y$ is an optimal solution of $\P$.
\end{proof}
Theorem \ref{th-ideal} states that the ideal inverse optimization problem is well-defined. However, even though
$\psd_{d_0}(\K)$ is a finite dimensional convex cone, it has no simple and tractable characterization to be used for practical 
computation. Therefore one needs an alternative and more tractable version of problem $\mathcal{P}$. Fortunately,
we next show that in the polynomial context such a formulation exists, thanks to the powerful Putinar's Positivstellensatz (Theorem \ref{th-put} above).

\section{Main result}

As the ideal inverse problem is intractable, we here provide tractable formulations
whose size depends on a parameter $d\in\N$. If the polynomial $\tilde{f}^*$ in Theorem \ref{th-ideal} belongs to $Q(g)$ then when $d$
increases the associated optimal value $\rho^k_d$ converges in finitely many steps to the optimal value $\rho^k$ of the ideal problem (\ref{inv-0}),
and $\tilde{f}^*$ can be obtained by solving finitely many semidefinite programs. And in fact this situation is {\it generic}.\\

With no loss of generality, i.e., up to some change of variable $\x'=\x-\y$ 
we may and will assume that $\y=0\in\K$. 

\subsection{A practical inverse problem}

With $d\in\N$ fixed, consider the following optimization problem $\P_d$:
\begin{equation}
\label{inv-1}
\begin{array}{rl}
\P_d:\quad\rho_d^k:=\displaystyle\min_{\tilde{f},\sigma_j\in\R[\x] }&\Vert f-\tilde{f}\Vert_k\\
\mbox{s.t.}& \tilde{f}(\x)-\tilde{f}(0)\,=\,\displaystyle\sum_{j=0}^{m} \sigma_j(\x)\,g_j(\x),\quad\forall\x\in\R^n\\
&\tilde{f}\in\R[\x]_{d_0};\:\sigma_j\in\Sigma[\x]_{d-v_j},\quad j=0,1,\ldots,m,
\end{array}
\end{equation}
where $d_0={\rm deg}\,f$, and $v_j=\lceil ({\rm deg}\,g_j)/2\rceil$, $j=1,\ldots,m$. 

The parameter $d\in\N$ impacts (\ref{inv-1}) by the maximum degree allowed for
the s.o.s. weights $(\sigma_j)\subset\Sigma[\x]$ in Putinar's certificate for the polynomial 
$\x\mapsto \tilde{f}(\x)-\tilde{f}(0)$, and so the higher $d$ is, the lower $\rho_d^k$.
Next, observe that in (\ref{inv-1}), the constraint
\[\tilde{f}(\x)-\tilde{f}(0)\,=\,\displaystyle\sum_{j=0}^{m} \sigma_j(\x)\,g_j(\x),\quad\forall\x\in\R^n,\]
is equivalent to stating that 
$\tilde{f}(\x)-\tilde{f}(0)\in Q^{d_0}_d(g)$, with $Q^{d_0}_d(g)$ as in (\ref{puT-suff1}).
Therefore, in particular, $\tilde{f}(\x)\geq\tilde{f}(0)$ for all $\x\in\K$, and so $0$ is a global minimizer of $\tilde{f}$ on $\K$.
So $\P_d$ is a strengthtening of $\mathcal{P}$ in that one has replaced the constraint
$\tilde{f}-\tilde{f}(0)\in \psd_{d_0}(\K)$ with the stronger condition $\tilde{f}-\tilde{f}(0)\in Q_d^{d_0}(g)$.
And so $\rho^k\leq\rho^k_d$ for all $d\in\N$. However, as we next see,
$\P_d$ is a tractable optimization problem with nice properties. Indeed,
$\P_d$ is a convex optimization problem and even a semidefinite program. 
For instance, if $k=1$ one may rewrite $\P_d$ as:
 
 \begin{equation}
\label{inv-3-k=1}
\begin{array}{rl}
\rho_d^1:=\displaystyle\min_{\lambda_\alpha\geq0,\tilde{f},\Z_j}&\displaystyle\sum_{\alpha\in\N^n_{d_0}\setminus\{0\}}\lambda_\alpha\\
\mbox{s.t.}& \lambda_\alpha +\tilde{f}_\alpha\geq f_\alpha,\quad\forall\alpha\in\N^n_{d_0}\setminus\{0\}\\
& \lambda_\alpha -\tilde{f}_\alpha\geq -f_\alpha,\quad\forall\alpha\in\N^n_{d_0}\setminus\{0\}\\
&\la\Z_0,\B_\alpha\ra+\displaystyle\sum_{j=1}^{m} \la\Z_j,\C^j_\alpha\ra\,=\,\left\{\begin{array}{l}
\tilde{f}_\alpha,\:\mbox{ if }0<\vert\alpha\vert\leq\,d_0\\
0,\:\mbox{ if $\alpha=0$ or $\vert\alpha\vert>d_0$}\end{array}\right.\\
&\Z_j\succeq0,\quad j=0,1,\ldots,m.\end{array}
\end{equation}
with $\B_\alpha$ as in (\ref{balpha}) and $\C^j_\alpha$ as in (\ref{calpha}) (with $g_j$ in lieu of $g$).
If $k=\infty$ then one may rewrite $\P_d$ as:
 
\begin{equation}
\label{inv-3-k=infty}
\begin{array}{rl}
\rho_d^\infty:=\displaystyle\min_{\lambda\geq0,\tilde{f},\Z_j}&\lambda\\
\mbox{s.t.}& \lambda +\tilde{f}_\alpha\geq f_\alpha,\quad\forall\alpha\in\N^n_{d_0}\setminus\{0\}\\
& \lambda -\tilde{f}_\alpha\geq -f_\alpha,\quad\forall\alpha\in\N^n_{d_0}\setminus\{0\}\\
&\la\Z_0,\B_\alpha\ra+\displaystyle\sum_{j=1}^{m} \la\Z_j,\C^j_\alpha\ra\,=\,\left\{\begin{array}{l}
\tilde{f}_\alpha,\:\mbox{ if }0<\vert\alpha\vert\leq\,d_0\\
0,\:\mbox{ if $\alpha=0$ or $\vert\alpha\vert>d_0$}\end{array}\right.\\
&\Z_j\succeq0,\quad j=0,1,\ldots,m,\end{array}
\end{equation}
and finally, if $k=2$ then one may rewrite $\P_d$ as:
\begin{equation}
\label{inv-3-k=2}
\begin{array}{rl}
\rho_d^2:=\displaystyle\min_{\lambda\geq0,\tilde{f},\Z_j}&\displaystyle\sum_{\alpha\in\N^n_{d_0}\setminus\{0\}}\lambda_\alpha\\
\mbox{s.t.}&\left[\begin{array}{cc}\lambda_\alpha &\tilde{f}_\alpha-f_\alpha\\
\tilde{f}_\alpha-f_\alpha&1\end{array}\right]\succeq0,\quad\forall\alpha\in\N^n_{d_0}\setminus\{0\}\\
&\la\Z_0,\B_\alpha\ra+\displaystyle\sum_{j=1}^{m} \la\Z_j,\C^j_\alpha\ra\,=\,\left\{\begin{array}{l}
\tilde{f}_\alpha,\:\mbox{ if }0<\vert\alpha\vert\leq\,d_0\\
0,\:\mbox{ if $\alpha=0$ or $\vert\alpha\vert>d_0$}\end{array}\right.\\
&\Z_j\succeq0,\quad j=0,1,\ldots,m.\end{array}
\end{equation}
\begin{rem}
\label{rem-f(0)}
{\rm Observe that in any feasible solution $(\tilde{f},\lambda,\Z_j)$ in all formulations
(\ref{inv-3-k=1})-(\ref{inv-3-k=2}),
$\tilde{f}_0$ plays  no role in the constraints of (\ref{inv-1}), but since we minimize
$\Vert \tilde{f}-f\Vert_k$ then it is always optimal to set $\tilde{f}_0=f_0$.
That is, $\tilde{f}_d(0)=\tilde{f}_0=f_0=f(0)$.
}\end{rem}

\subsection*{Sparsity}
The semidefinite program (\ref{inv-1})-(\ref{inv-3-k=1}) has $m+1$ Linear Matrix Inequalities (LMI's) 
$\Z_j\succeq0$ of size $O(n^d)$, which limits its application to
problems $\P$ of modest size. However large scale problems usually exhibit 
sparsity patterns which sometimes can be exploited. For instance, in \cite{lassparse} we have provided
a specialized ``sparse" version of Theorem \ref{th-put}
 for problems with structured sparsity
as described in Waki et al. \cite{waki}. Hence, with this specialized version
of  Putinar's Positivstellensatz, one obtains
a {\it sparse} positivity certificate which when substituted in (\ref{inv-1}),
would permit to solve (\ref{inv-1}) for problems of much larger size. Typically, in \cite{waki}
the authors have applied the ``sparse semidefinite relaxations" to problem $\P$ with up
to 1000 variables! Moreover, the {\it running intersection property} that must satisfy the sparsity pattern
for convergence guarantee of such relaxations \cite{lassparse}, is {\it not} needed in the present context of inverse optimization. This is because one imposes $\tilde{f}$ to satisfy this specialized Putinar's Positivstellensatz.

\subsection{Duality}

The semidefinite program dual of (\ref{inv-3-k=1}) reads

\[\left\{\begin{array}{ll}
\displaystyle\max_{\u,\v\geq0,\z}&\displaystyle
\sum_{\alpha\in\N^n_{d_0}\setminus\{0\}}f_\alpha(u_\alpha-v_\alpha)\:(=L_\z(f(0)-f))\\
\mbox{s.t.}&u_\alpha+v_\alpha\,\leq \,1,\quad\forall\,\alpha\in\N^n_{d_0}\setminus\{0\}\\
&u_\alpha-v_\alpha+z_\alpha \,=\,0,\quad\forall\alpha\in\N^n_{d_0}\setminus\{0\}\\
&\sum_{\alpha\in\N^n_d}z_\alpha\B_\alpha\succeq0\\
&\sum_{\alpha\in\N^n_d}z_\alpha\C^j_\alpha\succeq0,\quad j=1,\ldots,m,
\end{array}\right.\]
which, recalling the respective definitions (\ref{balpha}) and (\ref{calpha}) of the moment and localizing matrix, is the same as
\begin{equation}
\label{dual-k=1}
\left\{\begin{array}{ll}
\displaystyle\max_{\u,\v\geq0,\z}&\displaystyle
\sum_{\alpha\in\N^n_{d_0}\setminus\{0\}}f_\alpha(u_\alpha-v_\alpha)\:(=L_\z(f(0)-f))\\
\mbox{s.t.}&u_\alpha+v_\alpha\,\leq \,1,\quad\forall\,\alpha\in\N^n_{d_0}\setminus\{0\}\\
&u_\alpha-v_\alpha+z_\alpha \,=\,0,\quad\forall\alpha\in\N^n_{d_0}\setminus\{0\}\\
&\M_d(\z),\,\M_{d-v_j}(g_j\,\z)\,\succeq\,0,\quad j=1,\ldots,m.
\end{array}\right.
\end{equation}

Similarly, the semidefinite program dual of (\ref{inv-3-k=infty}) reads
\begin{equation}
\label{dual-k=infty}
\left\{\begin{array}{ll}
\displaystyle\max_{\u,\v\geq0,\z}&\displaystyle
\sum_{\alpha\in\N^n_{d_0}\setminus\{0\}}f_\alpha(u_\alpha-v_\alpha)\:(=L_\z(f(0)-f))\\
\mbox{s.t.}&\displaystyle\sum_{\alpha\in\N^n_{d_0}\setminus\{0\}}u_\alpha+v_\alpha\,\leq \,1\\
&u_\alpha-v_\alpha+z_\alpha\,=\,0,\quad\forall\alpha\in\N^n_{d_0}\setminus\{0\}\\
&\M_d(\z),\,\M_{d-v_j}(g_j\,\z)\,\succeq\,0,\quad j=1,\ldots,m,
\end{array}\right.
\end{equation}
and the semidefinite program dual of (\ref{inv-3-k=2}) reads
\begin{equation}
\label{dual-k=2}
\left\{\begin{array}{ll}
\displaystyle\max_{\z,\Delta_\alpha}&\displaystyle
\sum_{\alpha\in\N^n_{d_0}\setminus\{0\}}\left\la \Delta_\alpha,\left[\begin{array}{cc}0 &f_\alpha\\f_\alpha&-1\end{array}\right]\right\ra\\
\mbox{s.t.}&\displaystyle
\left\la \Delta_\alpha,\left[\begin{array}{cc}1 &0\\0&0\end{array}\right]\right\ra
\leq \,1,\quad\forall\alpha\in\N^n_{d_0}\setminus\{0\}\\
&\\
&\left\la \Delta_\alpha,\left[\begin{array}{cc}0 &1\\1&0\end{array}\right]\right\ra
+z_\alpha \,=\,0,\quad\forall\alpha\in\N^n_{d_0}\setminus\{0\}\\
&\M_d(\z),\,\M_{d-v_j}(g_j\,\z)\,\succeq\,0,\quad j=1,\ldots,m\\
&\Delta_\alpha\succeq0,\quad \forall\alpha\in\N^n_{d_0}\setminus\{0\}.\\
\end{array}\right.
\end{equation}
One may show that one may replace the criterion in (\ref{dual-k=2}) 
with the equivalent concave criterion
\[\max_{\z}\:\left\{L_\z(f(0)-f))-\frac{1}{4}\sum_{\alpha\in\N^n_{d_0}\setminus\{0\}}z_\alpha^2\:\right\}.\]

\begin{lem}
\label{nodualgap}
Assume that $\K\subset\R^n$ has nonempty interior. Then there is no duality gap between the semidefinite programs (\ref{inv-3-k=1}) and (\ref{dual-k=1}), (\ref{inv-3-k=infty}) and (\ref{dual-k=infty}),
and (\ref{inv-3-k=2}) and (\ref{dual-k=2}). Moreover, all semidefinite programs
(\ref{inv-3-k=1}), (\ref{inv-3-k=infty}) and (\ref{inv-3-k=2}) have an optimal solution $\tilde{f}_d\in\R[\x]_{d_0}$.
\end{lem}
\begin{proof}
The proof is detailed for the case $k=1$ and omitted for the cases $k=2$ and $k=\infty$ because
it is very similar.
Observe that $\rho_d^1\geq0$ and the constant polynomial $\tilde{f}(\x)=0$ for all $\x\in\R^n$, is an obviously feasible solution of
(\ref{inv-1}) (hence of (\ref{inv-3-k=1})).
Therefore $\rho_d^1$ being finite, it suffices to prove that Slater's condition\footnote{Slater's condition holds if there exists a strictly feasible solution, and so
for the dual (\ref{dual-k=1}), if there exists  $\z$ such that $\M_d(\z),\M_{d-v_j}(g_j\,\z)\succ0$, $j=1,\ldots,m$, and
$u_\alpha+v_\alpha<1,\,\forall\alpha\in\N^n_{d_0}\setminus\{0\}$. Then from a standard result in convex optimization, there is no duality gap between
(\ref{inv-3-k=1}) and (\ref{dual-k=1}), and if the values are bounded then (\ref{inv-3-k=1}) has an optimal solution.}
holds for the dual (\ref{dual-k=1}). Then the conclusion of Lemma
\ref{nodualgap} follows from a standard result of convex optimization.
 Let $\mu$ be the finite Borel measure defined by
  \[\mu(B)\,:=\, \int_{B\cap\K}\e^{-\Vert\x\Vert^2}d\x,\qquad \forall B\in\mathcal{B}\]
  (where $\mathcal{B}$ is the usual Borel $\sigma$-field), 
 and  let $\z=(z_\alpha)$, $\alpha\in\N^n_{2d}$, with
\[z_\alpha\,:=\,\kappa\,\int_\K \x^\alpha \,d\mu(\x),\qquad \alpha\in\N^n_{2d},\]
for some $\kappa>0$ sufficiently small to ensure that
\begin{equation}
\label{cond-1}
\kappa \,\displaystyle\left\vert \int \x^\alpha\,d\mu(\x)\right\vert \,<\,1,\qquad\forall \alpha\in\N^n_{2d}\setminus\{0\}.
\end{equation}

Define  $u_\alpha=\max[0,-z_\alpha]$ and $v_\alpha=\max[0,z_\alpha]$, 
$\alpha\in\N^n_{d_0}\setminus\{0\}$, so that
$u_\alpha+v_\alpha< 1$, $\alpha\in\N^n_{2d}\setminus\{0\}$. 
Hence $(u_\alpha,v_\alpha,\z)$ is a feasible solution
of (\ref{dual-k=1}). In addition, $\M_d(\z)\succ0$ and $\M_{d-v_j}(g_j\,\z)\succ0$, $j=1,\ldots,m$,
because $\K$ has nonempty interior, and so Slater's condition
holds for (\ref{dual-k=1}), the desired result.

If $k=\infty$ one chooses $\z$ such that
\[\kappa \,\displaystyle\sum_{\alpha\in\N^n_{d_0}\setminus\{0\}}\left\vert \int \x^\alpha\,d\mu(\x)\right\vert \,<\,1,\]
and if $k=2$ then one chooses $\z$ as in (\ref{cond-1}) 
and $\Delta_\alpha:=\left[\begin{array}{cc}1/2&\kappa_\alpha\\\kappa_\alpha &1\end{array}\right]\succ0$, for all 
$\alpha\in\N^n_{d_0}\setminus\{0\}$, such that
\[2\kappa_\alpha:=\kappa \,\displaystyle\int \x^\alpha\,d\mu(\x),\qquad\forall \alpha\in\N^n_{2d}\setminus\{0\}.\]
\end{proof}

\begin{thm}
\label{th1}
Assume that $\K$ in (\ref{setk}) has nonempty interior, and
let $\x^*\in\K$ be  a global minimizer of $\P$ with optimal value $f^*$,
and let $\tilde{f}_d\in\R[\x]_{d_0}$ be an optimal solution of $\P_d$ in (\ref{inv-1}) with optimal value $\rho_d^k$. Then:

{\rm (a)}  $0\in\K$ is a global minimizer of the problem $\tilde{f}^*_d=\min_\x\{\tilde{f}_d(\x)\,:\,\x\in\K\}$. In particular,
if $\rho_d^k=0$ then $\tilde{f}_d=f$ and $0$ is a global minimizer of $\P$.\\

{\rm (b)} If $k=1$ then $f^*\,\leq\,f(0)\,\leq\,f^*+\rho_d^1\,\displaystyle\sup_{\alpha\in\N^n_{d_0}} \vert (\x^*)^\alpha\vert$. In particular, if $\K\subseteq [-1,1]^n$ then
$f^*\,\leq\,f(0)\,\leq\,f^*+\rho_d^1$.\\

{\rm (c)} If $k=\infty$ then $f^*\,\leq\,f(0)\,\leq\,f^*+\rho_d^\infty\displaystyle\sum_{\alpha\in\N^n_{d_0}} \vert (\x^*)^\alpha\vert$. In particular if $\K\subseteq [-1,1]^n$ then
$f^*\,\leq\,f(0)\,\leq\,f^*+s(d_0)\,\rho_d^\infty$.
\end{thm}
\begin{proof}
(a) Existence of $\tilde{f}_d$ is guaranteed by Lemma \ref{nodualgap}. From the constraints of (\ref{inv-1}) we have:
$\tilde{f}_d(\x)-\tilde{f}(0)=\sum_{j=0}^{m}\sigma_j(\x)\,g_j(\x)$ 
which implies that
$\tilde{f}_d(\x)\geq \tilde{f}(0)$ for all $\x\in\K$, 
and so $0$ is a global minimizer of the optimization problem $\P':\:\min_\x\{\tilde{f}_d(\x)\,:\,\x\in\K\}$. 

(b) Let $\x^*\in\K$ be a global minimizer of $\P$. Observe that with $k=1$,
\begin{eqnarray}
\nonumber
f^*=f(\x^*)&=&\underbrace{f(\x^*)-\tilde{f}_d(\x^*)}+\underbrace{\tilde{f}_d(\x^*)-\tilde{f}_d(0)}_{\geq0}+\tilde{f}_d(0)\\
\nonumber
&\geq&\tilde{f}_d(0)-\vert \tilde{f}_d(\x^*)-f(\x^*)\vert\\
\nonumber
&\geq&\tilde{f}_d(0)-\Vert \tilde{f}_d-f\Vert_1\,\times\sup_{\alpha\in\N^n_{d_0}}\vert(\x^*)^\alpha\vert\\
\label{a}
&\geq&f(0)-\rho^1_d\,\sup_{\alpha\in\N^n_{d_0}}\vert(\x^*)^\alpha\vert
\end{eqnarray}
since $\tilde{f}_d(0)=f(0)$; see Remark \ref{rem-f(0)}.

(c) The proof is similar to that of (b) using that with $k=\infty$,
\[\vert\tilde{f}_d(\x)-f(\x)\vert \,\geq\,\left(\sup_{\alpha\in\N^n_{d_0}}\vert \tilde{f}_{d\alpha}-f_\alpha\vert\right)\times\sum_{\alpha\in\N^n_{d_0}}\vert\x^\alpha\vert.\]
\end{proof}
So not only Theorem \ref{th1} states that $0$ is the global optimum of the optimization problem
$\min\{\tilde{f}_d(\x)\,:\,\x\in\K\}$, but it also states that the optimal value $\rho^k_d$
also measures how far is $f(0)$ from the optimal value $f^*$ of the initial problem $\P$.
Moreover, observe that Theorem \ref{th1} merely requires existence of a minimizer
and nonemptyness of $\K$.  
In particular, $\K$  may not be compact.

\subsection{A canonical form for the $\ell_1$-norm}

When $\K$ is compact and if one uses the $\ell_1$-norm then the 
optimal solution $\tilde{f}_d\in\R[\x]_{d_0}$ in Theorem \ref{th1} (with $k=1$) takes a particularly
simple {\it canonical} form:

As $\K$ is compact we may and will assume (possibly after some scaling) that $\K\subseteq [-1,1]^n$
and so in the definition of (\ref{setk}) we may and will add the $n$ redundant quadratic constraints $g_{m+i}(\x)\geq0$, $i=1,\ldots,n$,
with $\x\mapsto g_{m+i}(\x)=1-x_i^2$ for every $i$, that is,
\begin{equation}
\label{setk-1}
\K=\{\x\in\R^n\::\:g_j(\x)\geq0,\quad j=1,\ldots,m+n\},\end{equation}
and
\[Q_d(g)\,=\,\left\{\sum_{j=0}^{n+m}\sigma_j\,g_j\::\quad\sigma_j\in\Sigma[\x]_{d-v_j},\:j=0,\ldots,m+n\right\},\]
which is obviously Archimedean.
\begin{thm}
\label{th1-ell1}
Assume that $\K$ in (\ref{setk-1}) has a nonempty interior and 
let $\tilde{f}_d\in\R[\x]_{d_0}$ be an optimal solution of $\P_d$ in (\ref{inv-1}) (with $m+n$ instead of $m$) with optimal value $\rho_d^1$
for the $\ell_1$-norm. Then:

(i) $\tilde{f}_d$ is of the form
\begin{equation}
\label{th1-ell1-1}
\tilde{f}_d(\x)\,=\,\left\{\begin{array}{ll}f(\x)+\b'\x&\mbox{if $d_0=1$}\\
f(\x)+\b'\x+\sum_{i=1}^n\lambda_i^*\,x_i^2&\mbox{if $d_0>1$,}\end{array}\right.
\end{equation}
for some vector $\b\in\R^n$ and some 
nonnegative vector $\lambda^*\in\R^n$, optimal solution of the semidefinite program:
 \begin{equation}
\label{inv-3-ell1}
\begin{array}{rl}
\rho_d^1:=\displaystyle\min_{\lambda,\b}&\displaystyle
\Vert\b\Vert_1+\sum_{i=1}^n\lambda_i\\
\mbox{s.t.}& f-f(0)+\b'\x+\displaystyle\sum_{i=1}^n\lambda_i\,x_i^2\,\in Q_{d}(g),\quad 
\lambda\geq0.
\end{array}
\end{equation}
(ii) The vector $\b$ is of the form $-\nabla f(0)+\sum_{j\in J(0)}\theta_j\,\nabla g_j(0)$ for some nonnegative scalars $(\theta_j)$
(where $j\in J(0)$ if and only if $g_j(0)=0$).
\end{thm}
\begin{proof}
(i) Notice that the dual (\ref{dual-k=1}) of (\ref{inv-3-k=1}) is equivalent to:
\begin{equation}
\label{equiv-1}
\left\{\begin{array}{ll}\displaystyle\max_{\z}&L_\z(f(0)-f))\\
\mbox{s.t.}&\M_d(\z),\,\M_{d-v_j}(g_j\,\z)\,\succeq\,0,\quad j=1,\ldots,m+n,\\
&\vert z_\alpha\vert \leq 1,\quad\forall\,\alpha\in\N^n_{d_0}\setminus\{0\}.
\end{array}\right.\end{equation}
Next, since $\M_{d}(\z)\succeq0$ one may invoke 
same arguments as those used in
Lasserre and Netzer \cite[Lemma 4.1, 4.2]{lass-netzer}, to
obtain that for very $\alpha\in\N^n_{2d}$ with $\vert\alpha\vert>1$,
\[\M_d(\z)\succeq0\,\Rightarrow\quad\vert z_\alpha\vert\leq \max_{i=1,\ldots,n}\{\max[L_\z(x_i^2),L_\z(x_i^{2d})]\}.\]
Moreover the constraint $\M_{d-1}(g_{m+i}\,\z)\succeq0$ implies
$\M_{d-1}(g_{m+i}\,\z)(\ell,\ell)\geq0$ for all $\ell$, and so in particular, one obtains
$L_\z(x_i^{2k-2})\geq L_\z(x_i^{2k})$ for all $k=1,\ldots,d$ and all $i=1,\ldots,n$.
Hence $\vert z_\alpha\vert\leq \max_{i=1,\ldots,n}L_\z(x_i^2)$ for every $\alpha\in\N^n_{2d}$ with $\vert\alpha\vert>1$.
Therefore  in (\ref{equiv-1}) one may replace the constraint $\vert z_\alpha\vert\leq 1$ for all $\alpha\in\N^n_{d_0}\setminus\{0\}$ with
the $2n$ inequality constraints $\pm L_\z(x_i)\leq1$, $i=1,\ldots,n$, if $d_0=1$ and
the $3n$ inequality constraints:
\begin{equation}
\label{simplify}
\pm L_\z(x_i)\leq 1,\:L_\z(x_i^2)\leq1,\:i=1,\ldots,n\end{equation}
if $d_0>1$.
Consequently, (\ref{equiv-1}) is the same as the semidefinite program
\begin{equation}
\label{equiv-1111}
\left\{\begin{array}{ll}\displaystyle\max_{\z}&L_\z(f(0)-f))\\
\mbox{s.t.}&\M_d(\z),\,\M_{d-v_j}(g_j\,\z)\,\succeq\,0,\quad j=1,\ldots,m+n,\\
&\pm L_\z(x_i)\leq 1,\:L_\z(x_i^2)\leq1,\:i=1,\ldots,n.
\end{array}\right.\end{equation}
Let $\b^1=(b^1_i)$ (resp. $\b^2=(b^2_i)$) be the nonnegative vector of dual variables associated with the constraints
$L_\z(x_i)\leq 1$ (resp. $-L_\z(x_i)\leq 1$), $i=1,\ldots,n$. Similarly, let $\lambda_i$ be the dual variable associated with the constraint $L_z(x_i^2)\leq 1$. Then the dual of (\ref{equiv-1111}) is the semidefinite program:
\[\left\{\begin{array}{ll}\displaystyle\max_{\b^1,\b^2,\lambda}&\displaystyle\sum_{i=1}^n\left((b_i^1+b_i^2)+\lambda_i\right)\\
\mbox{s.t.}&f-f(0)+(\b^1-\b^2)'\x+\displaystyle\sum_{i=1}^n\lambda_i\, x_i^2\,\in\,Q_{d}(g)\\
&\b^1,\b^2,\lambda\geq0
\end{array}\right.\]
which is equivalent to (\ref{inv-3-ell1}).

(ii) Let $(\b,\lambda)$ be an optimal solution of (\ref{inv-3-ell1}), so that
\[f-f(0)+\b'\x+\displaystyle\sum_{i=1}^n\lambda_i\, x_i^2=\sigma_0+\sum_{j=1}^{m+n}\sigma_j\,g_j,\]
for some SOS polynomials $\sigma_j$. Evaluating at $\x=0$ yields
\[\sigma_0(0)=0;\quad \underbrace{\sigma_j(0)}_{\theta_j\geq0}\,g_j(0)=0,\quad j=1,\ldots, m+n.\]
Differentiating and evaluating at $\x=0$ and using that $\sigma_j$ is SOS and $\sigma_j(0)g_j(0)=0$, $j=1,\ldots,n+m$, yields:
\[\nabla f(0)+\b\,=\,\sum_{j=1}^{n+m}\sigma_j(0)\,\nabla g_j(0)\,=\,\sum_{j\in J(0)}\theta_j\,\nabla g_j(0),\]
which is the desired result.
\end{proof}
From the proof of Theorem \ref{th1-ell1}, this special form of $\tilde{f}_d$ is specific to the $\ell_1$-norm,
which yields the constraint $\vert z_\alpha\vert\leq1,\:\alpha\in\N^n_{d_0}\setminus\{0\}$ in the dual (\ref{dual-k=1}) and allows its simplification
(\ref{simplify}) thanks to a property of the moment matrix described in \cite{lass-netzer}. Observe that the canonical form (\ref{th1-ell1-1}) of $\tilde{f}_d$ is {\it sparse}
since $\tilde{f}_d$ differs from $f$ in at most $2n$ entries only (recall that $f$ has ${n+d_0\choose n}$ entries). This is another example of sparsity properties 
of optimal solutions of $\ell_1$-norm minimization problems, already observed in other contexts (e.g., in some compressed sensing applications). Moreover, it has the following consequence for nonlinear $0/1$ programs.

\begin{cor}
\label{0/1}
Let $\K=\{0,1\}^n$, $f\in\R[\x]_d$ and let $\y\in\K$ with 
$I_1:=\{i:y_i=0\}$ and $I_2:=\{i:y_i=1\}$.
Then an optimal solution $\tilde{f}_d\in\R[\x]_{d_0}$ of the inverse problem (\ref{inv-1}) for
the $\ell_1$-norm, is of the form
\begin{equation}
\label{form-01}
\tilde{f}_d(\x)\,=\,f(\x)+\b^T(\x-\y),\end{equation}
for some coefficient vector $\b\in\R^n$ such that $b_i\geq0$ if $i\in I_1$
and $b_i\leq0$ if $i\in I_2$, $i=1,\ldots,n$. Moreover, $\b$ is an optimal solution of the semidefinite program:
\[\begin{array}{ll}\displaystyle\min_{\b,\lambda,\gamma}&\displaystyle\sum_{i\in I_1}b_i-\sum_{i\in I_2}b_i\:(=\vert\b\vert)\\
\mbox{s.t.}&f(\bx)-f(\y)+\b^T (\bx-\y)=\sigma_0+\displaystyle\sum_{i=1}^n(x_i^2-x_i)\,\sigma_i\\
&b_i\geq0,\:i\in I_1;\:b_i\leq0,\:i\in I_2;\quad\sigma_0\in\Sigma[\x]_d;\:\sigma_i\in\R[\x]_{d-1},\,i=1,\ldots,n.
\end{array}\]
\end{cor}
\begin{proof}
We briefly sketch the proof which 
is very similar to that of Theorem \ref{th1-ell1} even though $\K$ does not have a nonempty interior 
and $d_0$ is not required to be even. Recall that $\y\in\K$ could be assumed to be $0$ and so if $\K=\{0,1\}^n$
then the new feasible set after the change of variable $\u:=\x-\y$ is now
$\tilde{\K}=\prod_{i\in I_1}(\{0,1\})\prod_{i\in I_2}(\{-1,0\})$. Similarly, let $f'\in\R[\bx]_d$ be the polynomial $f$ in the new coordinates $\u$,
i.e., $f'(\u)=f(\u+\y)$.
So for every $\alpha\in\N^n$, define $\bar{\alpha}\in\{0,1\}^n$ and $\Delta_{\alpha}\subset I_2$ by: 
\begin{eqnarray*}
\bar{\alpha}_i&:=&1\mbox{ if $\alpha_i>0$ and $0$ otherwise, $i=1,\ldots,n$}\\ 
\Delta_{\alpha}&:=&\{i\in I_2:0<\alpha_i\mbox{ is even}\},
\end{eqnarray*}
and let $\vert\Delta_{\alpha}\vert$ denotes the cardinality of $\Delta_{\alpha}$. 
Then because of the boolean constraints $u_i^2=u_i$, $i\in I_1$ and $u_i^2=-u_i$, $i\in I_2$,
in the definition of $\tilde{\K}$,
(\ref{equiv-1}) reads
\[\left\{\begin{array}{ll}\displaystyle\max_{\z}&L_\z(f'(0)-f')\\
\mbox{s.t.}&\M_d(\z)\succeq\,0;\quad z_{\alpha}=(-1)^{\vert\Delta_{\alpha}\vert}\,z_{\bar{\alpha}},\quad\alpha\in\N^n_{2d}\\
&\vert z_{\alpha}\vert \leq 1,\quad\forall\,\alpha\in\N^n_{d_0}\setminus\{0\}.
\end{array}\right.\]
But this combined with $\M_d(\z)\succeq0$ implies that $\M_d(\z)$ 
can be simplified to a smaller real symmetric matrix $\overline{\M}_d(\z)$ with rows and columns indexed by only the square-free 
monomials $\x^{\bar{\alpha}}$, $\alpha\in\N^n_{2d}$.
Indeed, every column $\alpha$ of $\M_d(\z)$ is exactly identical to $\pm$ the column of $\M_d(\z)$  indexed by $\bar{\alpha}$.
Also, invoking \cite{lass-netzer} and with similar arguments,
$\vert z_{\alpha}\vert\leq \max\left[\max_{i\in I_1}L_\z(u_i),\,\max_{i\in I_2}-L_\z(u_i)\right]$ 
for all $\alpha$, and so (\ref{equiv-111}) is equivalent to:
\begin{equation}
\label{equiv-111}
\left\{\begin{array}{ll}\displaystyle\max_{\z}&L_\z(f'(0)-f')\\
\mbox{s.t.}&\overline{\M}_d(\z)\succeq\,0;\quad z_{\alpha}=(-1)^{\vert\Delta_{\alpha}\vert}\,z_{\bar{\alpha}},\quad\alpha\in\N^n_{2d}\\
&L_\z(\x_i)\,\leq\,1,\qquad i\in I_1\\
&-L_\z(\x_i)\,\leq\,1,\qquad i\in I_2.
\end{array}\right.\end{equation}
Finally, let $\mu$ be a Borel measure with support exactly $\tilde{\K}$, and scaled to satisfy
$\vert\int u_id\mu\vert <1$, $i=1,\ldots, n$. Its associated vector of moment $\z=(\int \u^{\alpha} d\mu)$, $\alpha\in\N^n$, is feasible in
(\ref{equiv-111}) and $\overline{\M}_d(\z)\succ0$. Hence Slater's condition holds for (\ref{equiv-111}), which in turn
implies that there is no duality gap with its dual which reads:
\[\begin{array}{ll}\displaystyle\min_{\b,\lambda,\gamma}&\displaystyle\sum_{i=1}^nb_i\\
\mbox{s.t.}&f'(\u)-f'(0)+\displaystyle\sum_{i\in I_1}b_iu_i-\displaystyle\sum_{i\in I_2}b_iu_i=\sigma_0(\u)\\
&+\displaystyle\sum_{i\in I_1}^n\sigma_i(\u)\,(u_i^2-u_i)
+\displaystyle\sum_{i\in I_2}^n\sigma_i(\u)\,(u_i^2+u_i)\\
&\b\geq0;\quad\sigma_0\in\Sigma[\u]_d;\:\sigma_i\in\R[\u]_{d-1},\,i=1,\ldots,n.
\end{array}\]
Moreover, the dual has an optimal solution because the optimal value 
is bounded below by zero. Recalling that $\u=\x-\y$ and $f'(\u)=f(\x)$, we retrieve the  semidefinite program of the Corollary
and so $\tilde{f}_d$ is indeed of the form (\ref{form-01}).
\end{proof}

\subsection{Structural constraints}
\label{structural}
It may happen that the initial criterion $f\in\R[\x]$ has some structure that one wishes to keep
in the inverse problem. For instance, in MAXCUT problems on $\K=\{-1,1\}^n$, $f$ is a quadratic form
$\x\mapsto \x' \A\x$ for some real symmetric matrix $\A$ associated with a graph $(V,E)$, where $\A_{ij}\neq0$ if 
and only if $(i,j)\in E$. Therefore, in the inverse optimization problem, one may wish that $\tilde{f}$ in (\ref{inv-1}) 
is also a quadratic form associated with the same graph $(V,E)$, so that $\tilde{f}(\x)=\x'\tilde{\A}\x$ with
$\tilde{\A}_{ij}=0$ for all $(i,j)\not\in E$.

So if $\Delta_f\subset\N^n_{d_0}$ denotes the subset of (structural) multi-indices for which $f$ and $\tilde{f}$ should have same coefficient,
then in (\ref{inv-1}) one includes the additional constraint $\tilde{f}_\alpha=f_\alpha$ for all $\alpha\in\Delta_f$. Notice that $0\in\Delta_f$ because $\tilde{f}_0=f_0$; see Remark \ref{rem-f(0)}.
For instance, with $\K$ as in (\ref{setk-1}) and $k=1$, (\ref{inv-3-k=1}) reads
 \begin{equation}
\label{inv-33-k=1}
\begin{array}{rl}
\rho_d^1:=\displaystyle\min_{\tilde{f},\lambda_\alpha,\Z_j}&\displaystyle\sum_{\alpha\in\N^n_{d_0}\setminus\{0\}}\lambda_\alpha\\
\mbox{s.t.}& \lambda_\alpha +\tilde{f}_\alpha\geq f_\alpha,\quad\forall\alpha\in\N^n_{d_0}\setminus\Delta_f\\
& \lambda_\alpha -\tilde{f}_\alpha\geq -f_\alpha,\quad\forall\alpha\in\N^n_{d_0}\setminus\Delta_f\\
&\la\Z_0,\B_\alpha\ra+\displaystyle\sum_{j=1}^{m+n} \la\Z_j,\C^j_\alpha\ra\,=\,\left\{\begin{array}{l}
f_\alpha,\:\mbox{ if }\alpha\in\Delta_f\setminus\{0\}\\
\tilde{f}_\alpha,\:\mbox{ if }\alpha\in\N^n_{d_0}\setminus \Delta_f\\
0,\:\mbox{ if $\alpha=0$ or $\vert\alpha\vert>d_0$}\end{array}\right.\\
&\Z_j\succeq0,\quad j=0,1,\ldots,m+n,\end{array}
\end{equation}
and its dual has the equivalent form,
\begin{equation}
\label{equiv-11}
\left\{\begin{array}{ll}\displaystyle\max_{\z}&L_\z(f(0)-f))\\
\mbox{s.t.}&\M_d(\z),\,\M_{d-v_j}(g_j\,\z)\,\succeq\,0,\quad j=1,\ldots,m+n,\\
&\vert z_\alpha\vert \leq 1,\quad\forall\,\alpha\in\N^n_{d_0}\setminus(\Delta_f\cup\{0\}).
\end{array}\right.\end{equation}

However now (\ref{inv-33-k=1}) may not have a feasible solution. In problems where $d_0$ is even and $\Delta_f$ does not contain 
the monomials $\alpha\in\N^n_{d_0}$ such that $\x^\alpha=x_i^2$, or $\x^\alpha=x_i^{d_0}$, $i=1,\ldots,n$, and if (\ref{inv-33-k=1}) has a feasible solution, then there is an optimal solution $\tilde{f}$ with still the special form described in Theorem \ref{th1-ell1}, but with $b_k=0$ if $\alpha=e_k\in\Delta_f$ (where all entries of
$e_k$ vanish except the one at position $k$).

\subsection{Asymptotics when $d\to\infty$}

We now relate $\P_d$, $d\in\N$, with the ideal inverse problem $\mathcal{P}$ in (\ref{inv-0}) when $d$ increases.

\begin{prop}
\label{prop-asymptotics}
Let $\K$ be as in (\ref{setk}) with nonempty interior.
For every $k=1,2,\infty$,  let 
$\tilde{f}_d\in\R[\x]_{d_0}$ (resp. $\tilde{f}^*\in\R[\x]_{d_0}$) be an optimal solution of (\ref{inv-1}) (resp. (\ref{inv-0}))
with associated optimal value $\rho^k_d$ (resp. $\rho^k$).

The sequence $(\rho^k_d)$, $d\in\N$, is monotone nonincreasing and converges to $\hat{\rho}^k\geq\rho^k$.
Moreover, every accumulation point $\hat{f}\in\R[\x]_{d_0}$ of the sequence $(\tilde{f}_d)$, $d\in\N$, 
is such that $\hat{f}-\hat{f}(0)\in\psd_{d_0}(\K)$ and 
$\Vert \hat{f}-f\Vert_k=\hat{\rho}^k$. Finally, if
$\tilde{f}^*-\tilde{f}^*(0)$ is in $Q(g)$, then 
$\rho^k_d=\hat{\rho}^k=\rho^k$ for some $d$.
\end{prop}
\begin{proof}
Observe that the sequence $(\tilde{f}_d)$, $d\in\N$, is contained in the ball 
$\{h\,:\,\Vert h-f\Vert_k\leq \rho^k_{d_0}\}\subset\R[\x]_{d_0}$,
for some $d_0\in\N$. So let $\hat{f}$ be an accumulation point of $(\tilde{f}_d)$.
Since $\tilde{f}_d-\tilde{f}(0)\geq0$ on $\K$ for all $d$, a simple continuity argument
yields $\hat{f}-\hat{f}(0)\geq0$ on $\K$, i.e., $\hat{f}-\hat{f}(0)\in\psd_{d_0}(\K)$. Moreover, the sequence $(\rho^k_d)$ is obviously monotone nonincreasing
and bounded below by zero. Hence $\lim_{d\to\infty}\rho^k_d=:\hat{\rho}^k\geq\rho^k$, 
and by continuity $\Vert\hat{f}-f\Vert_k=\hat{\rho}^k$.

Finally, if $\tilde{f}^*-\tilde{f}(0)\in Q(g)$ then $\tilde{f}^*-\tilde{f}(0)\in Q^{d_0}_{d}(g)$ for some $d$, and so
$\tilde{f^*}$ is a feasible solution of (\ref{inv-1}) but with value $\rho^k\leq\rho^k_d$.
Therefore, we conclude that $\tilde{f}^*$ is an optimal solution of (\ref{inv-1}).
\end{proof}

Proposition \ref{prop-asymptotics} relates $\rho^k_d$ and $\rho^k$ in a strong sense when
$\tilde{f}^*-\tilde{f}(0)\in Q(g)$. However, we would like to know how restrictive is the constraint 
$\tilde{f}^*-\tilde{f}(0)\in Q(g)$ compared to $\tilde{f}^*-\tilde{f}(0)\in \psd_{d_0}(\K)$. Indeed,
even though $\psd_{d_0}(\K)={\rm cl}\,(\cup_{\ell=0}^\infty)Q^{d_0}_\ell$ when $\K$ satisfies the
assumptions of Theorem \ref{limit}, in general
an approximating sequence $(f_\ell)\subset Q(g)$, $\ell\in\N$ (with $\Vert f_\ell-\tilde{f}^*\Vert_k\to 0$), does
not have the property that
$f_\ell(\x)-f_\ell(0)\geq0$ for all $\x$ on $\K$. 

\subsection{$Q(g)$ versus $\psd_{d_0}(\K)$}

Therefore the question is: {\em How often a polynomial
nonnegative on $\K$ (and with at least one zero in $\K$) is an element of $Q(g)$?}
This question can be answered in a number of cases which suggest that
$f\geq0$ on $\K$ and $f\not\in Q(g)$ can be true in very specific cases only (at least when $\K$ is compact
and $Q(g)$ is Archimedean).
Indeed $f\geq0$ on $\K$ implies $f\in Q(g)$ whenever:
\begin{itemize}
\item $f$ and $-g_j$ are convex, $j=1,\ldots,n$,
Slater's condition holds and $\nabla^2f(\x^*)\succ0$ at the unique global minimizer $\x^*\in\K$; see e.g.,
de Klerk and Laurent \cite{deklerk-laurent}.
\item $\K\subseteq\{0,1\}^n$, i.e., for 0/1 polynomial programs, and more generally for
all discrete polynomial optimization problems.
\item $Q(g)$ is Archimedean, $f$ has finitely many zeros in $\K$ and the {\em Boundary Hessian Condition} (BHC) holds 
at every zero of $f$ in $\K$; see e.g. Marshall \cite{marshall2}. That is, the BHC holds at a zero $\x^*\in\K$
of $f$ if there exists $0\leq k\leq n$, and some $1\leq v_1\leq \cdots\leq v_k\leq m$, such that:
 $g_{v_1},\ldots, g_{v_k}$ are part of a system of local parameters at $\x^*$, and  the standard sufficient conditions for $\x^*$ to be local 
 minimizer of $f$ on $\{\x:g_{v_j}(\x)\geq0,\:j=1,\ldots,k\}$ hold. Equivalently, if 
 $(t_1,\ldots,t_n)$ are local parameters at $\x^*$ with $t_j=g_{v_j}$, $j=1,\ldots,k$,
then $f$ can be written as a formal power series $f_0+f_1+f_2\ldots$ in $\R[[t_1,\ldots,t_n]]$, where
each $f_j$ is an homogeneous form of degree $j$, 
 \[f_1\,=\,a_1\,t_1+\cdots +a_k\,t_k\quad\mbox{ with }a_i>0,\quad i=1,\ldots,k,\]
and the quadratic form $f_2(\underbrace{0,\ldots, 0}_{k \mbox{ times}},t_{k+1},\ldots,t_n)$ is positive definite.

For instance, if all the (finitely many) zeros $\x^*$ of $f$  are in the interior of $\K$, one may take
$k=0$ and $t_i(\x)=x_i-x^*_i$ for all $i$. Then $f\in Q(g)$ if $f_2$ is positive definite.
\end{itemize}
It turns out that under a technical condition on the polynomials that define $\K$, the BHC holds generically, i.e.,
the set of polynomials $f\in\R[\x]_d$ for which the BHC holds at every global minimizer on $\K$,
is dense in $\R[\x]_d$; see Marshall \cite[Corollary 4.5]{marshall2}. Finally, and in the same vein,
a recent result of Nie \cite{nie} states that if 
the standard constraint qualification, strict complementarity and second-order sufficiency condition hold
at every global minimizer of $f$ on $\K$, then $f-f^*\in Q(g)$. Moreover this property is also generic in the sense that
it does not hold only if the coefficients of the polynomials $f$ and $g_j$, $j=1,\ldots,m$,
satisfy a set of polynomial equations! In other words the three conditions
hold in a Zariski open set;  for more details see Nie \cite[Theorem 1.1, Theorem 1.2 and \S 4.2]{nie}. 

So in view of Proposition \ref{prop-asymptotics}, one may expect that $\lim_{d\to\infty}\hat{\rho}^k_d=\rho^k$ generically. On the other hand,
given $d\in\N$, identifying whether  $\rho^k_d=\rho^k$ (or whether $\vert \rho^k_d-\rho^k\vert<\epsilon$ for some given $\epsilon>0$, fixed) is a open issue. For instance, $\rho^k_d=\rho^k_{d+1}$ (as is the case in Examples \ref{newex1} and \ref{newex2} below) is not a guarantee that $\rho^k_d=\rho^k$.

We will also see in Section \S \ref{section-epsilon} that one may also approach 
as closely as desired an optimal  solution of the ideal inverse problem (\ref{inv-0}) by asking $\y$ to be only 
a global $\epsilon$-minimizer (with $\epsilon>0$ fixed), provided that $\epsilon$ is small enough.

In addition, more can be said by looking at the dual of (\ref{inv-0}).

\subsection*{The dual of the ideal inverse problem $\mathcal{P}$}

We now provide an explicit interpretation of the dual problems $\P^*_d$ in 
(\ref{dual-k=1})-(\ref{dual-k=infty}).
Let $M(\K)$ be the space of finite Borel measures on $\K$.
Then obviously (\ref{dual-k=1}) is a relaxation of the following problem:
\begin{equation}
\label{newdual-k=1}
\left\{\begin{array}{ll}
\br^1=\displaystyle\max_{\mu\in M(\K)}&\displaystyle\int_\K (f(0)-f(\x))\,d\mu(\x)\\
\mbox{s.t.}&\pm\displaystyle\int_\K \x^\alpha d\mu(\x)\,\le\,1,\quad \forall\,\alpha\in\N^n_{d_0}\setminus\{0\},
\end{array}\right.
\end{equation}
which, denoting by $\delta_0$ the Dirac measure at $\x=0$, and by $P(\K)$ the space of Borel probability measures on $\K$, 
can be rewritten as
\[\left\{\begin{array}{rl}
\displaystyle\max_{\nu\in P(\K),\gamma\geq0}&\gamma\, (\delta_0(f)-\nu(f))\\
\mbox{s.t.}&
\pm\gamma\,\displaystyle (\nu(\x^\alpha)-\delta_0(\x^\alpha))\,\le\,1,\quad \forall\,\alpha\in\N^n_{d_0};\quad\nu(\K)=\gamma.
\end{array}\right.\]
Similarly, (\ref{dual-k=infty}) is a relaxation of the following problem:
\begin{equation}
\label{newdual-k=infty}
\br^\infty=\displaystyle\max_{\mu\in M(\K)}\,\left\{\displaystyle\int_\K (f(0)-f(\x))\,d\mu(\x)\::\:
\displaystyle\sum_{\alpha\in\N^n_{d_0}\setminus\{0\}}\left\vert\int_\K \x^\alpha d\mu\right\vert\,\le\,1\,\right\},
\end{equation}
or, again, equivalently,
\[\displaystyle\max_{\nu\in P(\K),\gamma\geq0}\,\left\{\gamma\,(\delta_0(f)-\nu(f))\::\:
\gamma\,\displaystyle\sum_{\alpha\in\N^n_{d_0}}\left\vert\nu(\x^\alpha)-\delta_0(\x^\alpha)\right\vert\,\le\,1;
\quad\nu(\K)=\gamma\,\right\}.\]
Hence, in the dual problems (\ref{newdual-k=1}) and (\ref{newdual-k=infty}) 
one searches for a finite Borel measure $\mu$ 
which concentrates as much as possible on the set $\{\x\in\K\,:\,f(\x)\leq f(0)\}$, and such
that  its moments up to order $d_0$ are not too far from those of a measure supported at $\{0\}\in\K$. 

In fact, and as one might have expected, (\ref{newdual-k=1}) (resp. 
(\ref{newdual-k=infty})) is the dual of $\mathcal{P} $ in (\ref{inv-0}) with $k=1$ (resp. with $k=\infty$).
For instance, with $k=1$, to see that weak duality holds,
let $\tilde{f}\in\R[\x]_{d_0}$  and $\mu\in M(\K)$ be
an arbitrary feasible solution of (\ref{inv-0}) and (\ref{newdual-k=1}), respectively. Then:
\begin{eqnarray*}\displaystyle\int_\K (f(0)-f)\,d\mu&=&
\underbrace{\displaystyle\int_\K (\tilde{f}(0)-\tilde{f})d\mu}_{\leq 0}+\int_\K (\tilde{f}-f)\,d\mu\\
&\leq&\displaystyle\sum_{\alpha\in\N^n_{d_0}\setminus\{0\}}\vert \tilde{f}_\alpha-f_\alpha\vert\cdot\left\vert\int_\K \x^\alpha\,d\mu(\x)\right\vert
\leq\displaystyle\Vert \tilde{f}-f\Vert_1,
\end{eqnarray*}
i.e., weak duality holds and $\br_1\leq\rho^1$.  We even have the following:

\begin{lem}
\label{lem-asymp}
Let $\K$ in (\ref{setk}) be with nonempty interior and assume that 
$Q(g)$ is Archimedean. Let $\rho^k$ be as in (\ref{inv-0}) with $k=1,\infty$, and
let $\z^d=(z^d_\alpha)\in\R^{s(d)}$ be a nearly optimal solution of (\ref{dual-k=1})
(or (\ref{equiv-1})), e.g., with value
$L_\z(f(0)-f)\geq \rho^k_d-1/d$, for all $d\in\N$. 

If $\displaystyle\liminf_{d\to\infty}z^d_0<\infty$ 
then $\displaystyle\lim_{d\to\infty}\rho^k_d=\rho^k$ and (\ref{newdual-k=1}) has an optimal solution
$\mu^*\in M(\K)$ which is supported on the set of global minimizers on $\K$ of the optimal solution $\tilde{f}^*\in\R[\x]_{d_0}$ of
(\ref{inv-0}) (which contains $\{0\}$). Hence either $\rho^k=0$ in which case $0$ is an optimal solution of $\P$, or
$\rho^k>0$ and $\tilde{f}^*$ has a another global minimizer $\tilde{\x}\neq0$ on $\K$ with $f(\tilde{\x})<f(0)$.
\end{lem}
\begin{proof}
The proof for the case $k=\infty$ is omitted as very similar to that of the case $k=1$.
Consider the subsequence $d_i$, $i\in\N$, such that $\displaystyle\liminf_{d\to\infty}z^d_0=\displaystyle\lim_{i\to\infty} z^{d_i}_0 <\infty$. Using the Archimedean property (\ref{archimedean}) of $Q(g)$, we proceed exactly as in the proof of Theorem 3.2 in \cite[p. 57--59 ]{lassjogo}.
There is a infinite sequence $\z^*=(z^*_\alpha)$, $\alpha\in\N^n$, and a subsequence (still denoted
$d_i$ for notational convenience), such that for every $\alpha\in\N^n$, $z^{d_i}_\alpha\to z^*_\alpha$.
Moreover, from the convergence $\z^{d_i}\to\z^*$, $\M_d(g_j\,\z^*)\succeq0$ for every $d\in\N$ and every
$ j=0,1,\ldots,m$; hence by Putinar's Theorem \cite{putinar}, $\z^*$ is the sequence of moments of a finite Borel measure $\mu^*$ supported on $\K$. Moreover,
since $\vert z^{d_i}_\alpha\vert\leq 1$ for all $\alpha\in\N^n_{d_0}\setminus\{0\}$, we obtain
\[\vert\int_\K\x^\alpha\,d\mu^*\vert\,=\,\vert z^*_\alpha\vert
=\lim_{i\to\infty}\vert z^{d_i}_\alpha\vert\,\leq \,1\quad\forall \alpha\in\N^n_{d_0}\setminus\{0\},\]
which proves that $\mu^*$ is feasible for (\ref{newdual-k=1}). Finally, by monotonicity of the
sequence $(\rho^1_d)$, $d\in\N$, and using $\rho^1_d\geq\rho^1\geq\br^1$ for all $d$,
\begin{eqnarray*}
\br^1\leq\rho^1\leq\lim_{d\to\infty}\rho^1_{d}&=&\lim_{i\to\infty}\rho^1_{d_i}\,=\,\lim_{i\to\infty}L_{\z^{d_i}}(f(0)-f)\\
&=&L_{\z^*}(f(0)-f)=\int_\K(f(0)-f)d\mu^*,\end{eqnarray*}
which proves that $\mu^*$ is an optimal solution of (\ref{newdual-k=1}), and so
$\br_1=\rho^1$.

Finally, since $\tilde{f}^*(0)=f(0)$,
\begin{eqnarray*}
\rho^1=\displaystyle\int_\K (f(0)-f)\,d\mu^*&=&\underbrace{\displaystyle\int_\K (\tilde{f}^*(0)-\tilde{f}^*)\,d\mu^*}_{\leq 0}
+\int_\K (\tilde{f}^*-f)\,d\mu^*\\
&\leq&\displaystyle\Vert \tilde{f}^*-f\Vert_1=\rho^1,
\end{eqnarray*}
which implies $\mu^*(\{\x:\,\tilde{f}^*(\x)-\tilde{f}^*(0)>0\})=0$, that is,
the support of $\mu^*$ is contained in the set of global minimizers of $\tilde{f}^*$ (which contains $\{0\}$).
Therefore, if $\rho^1>0$ then necessarily there is another global minimizer $0\neq\tilde{\x}\in\K$ of $\tilde{f}^*$
with $f(\tilde{\x})<f(0)$, otherwise $\rho^1=\int (f(0)-f)d\mu^*=0$.
\end{proof}

\subsection{Convexity}
One may wish to restrict to search for convex polynomials $\tilde{f}\in\R[\x]_{d_0}$ (no matter if $f$ itself is convex).
For instance if the $g_j$'s are concave (so that $\K$ is convex) but $f$ is not, one may wish to find the convex optimization problem
whose $\y\in\K$ is an optimal solution and with convex polynomial criterion $\tilde{f}\in\R[\x]_{d_0}$ closest to $f$.

If $d_0>2$ then in (\ref{inv-1}) it suffices to add the additional Putinar's certificate
\begin{equation}
\label{convex}
(\x,\u)\mapsto\:\u^T\nabla^2\tilde{f}(\x)\,\u\,=\,\sum_{j=0}^m\psi_j(\x,\u)\,g_j(\x)+\psi_{m+1}(\x,\u)(1-\Vert\u\Vert^2),\end{equation}
with $\psi_{m+1}\in\R[\x,\u]$ and 
$\psi_j\in\Sigma_{d-v_j}[\x,\u]$, for all $j=0,1,\ldots,m$. Indeed, (\ref{convex}) is a Putinar's certificate
of convexity for $\tilde{f}$ on $\K$, with degree bound $d$. As the coefficients of
the polynomial $(\x,\u)\mapsto\:\u^T\nabla^2\tilde{f}(\x)\u$ are linear in the coefficients of $\tilde{f}$,
(\ref{convex}) will translate into additional semidefinite constraints in (\ref{inv-3-k=1}).

If $d_0\leq2$, i.e. if $f(\x)=\frac{1}{2}\x^T\A\x+\b^T\x+c$ for some real symmetric matrix $\A\in\R^{n\times n}$, some
vector $\b\in\R^n$ and some scalar $c\in\R$, then
$\tilde{f}(\x)=\frac{1}{2}\x^T\tilde{\A}\x+\tilde{\b}^T\x+\tilde{c}$ for some real symmetric matrix 
$\tilde{\A}\in\R^{n\times n}$, some $\tilde{\b}\in\R^n$ and some $\tilde{c}\in\R$.
In that case, in (\ref{inv-1}) it suffices to add 
constraint $\nabla^2\tilde{f}(\x)=\tilde{\A}\succeq0$,
which is just a Linear Matrix Inequality (LMI).
And therefore, again, (\ref{inv-1}) can be rewritten
as a semidefinite program, namely (\ref{inv-3-k=1})-(\ref{inv-3-k=2}) with the additional LMI constraint 
$\tilde{\A}\succeq0$.

Notice that for $k=1,2$, it also makes sense to search for 
$\tilde{f}\in\R[\x]_2$ even if $f$ has degree $d_0>2$, i.e.,
if $f(\x)=c+\b^T\x+\frac{1}{2}\x^T\A\x +h(\x)$ where $h\in\R[\x]$ does not contains monomials
of degree smaller than $3$. This means that one searches for the convex program with quadratic cost
closest to $f$. 

So for instance, in the case where one searches for $\tilde{f}\in\R[\x]_2$,
and given $\y\in\K$ let 
$J(\y):=\{j\in\{1,\ldots,m\}\,:\,g_j(\y)=0\}$ be the set of constraints that are active at $\y$.
If the $g_j$'s that define $\K$ are concave 
then one may simplify (\ref{inv-1}). Writing $\tilde{f}=\frac{1}{2}\x^T\tilde{\A}\x+\tilde{\b}^T\x+\tilde{c}$, and with $k=1,2$,
(\ref{inv-1}) now reads:

\[\begin{array}{rl}
\rho^k:=\displaystyle\min_{\tilde{\A},\tilde{\b},\lambda}&\Vert f-\tilde{f}\Vert_k\\
\mbox{s.t.}& \tilde{\A}\,\y+\tilde{\b}\,=\,\displaystyle\sum_{j\in J(\y)}\lambda_j\,\nabla g_j(\y)\\
&\tilde{\A}\succeq0;\:\lambda_j\geq0,\quad j\in J(\y).
\end{array}\]
So, as we did in the previous section, and possibly after the change of variable $\x':=\x-\y$,
with no loss of generality one may and will assume that  $\y=0$, in which case (\ref{inv-4}) simplifies to
\begin{equation}
\label{inv-4}
\begin{array}{rl}\rho^k:=\displaystyle\min_{\tilde{\A},\tilde{\b},\lambda}&\Vert f-\tilde{f}\Vert_k\\
\mbox{s.t.}& \tilde{\b}\,=\,\displaystyle\sum_{j\in J(0)}\lambda_j\,\nabla g_j(0)\\
&\tilde{\A}\succeq0;\:\lambda_j\geq0,\quad j\in J(0),
\end{array}\end{equation}
which in turn simplifies to
\begin{equation}
\label{inv-44}
\begin{array}{rcl}
\rho^1&=&\displaystyle\min_{\tilde{\A}\succeq0}\Vert \tilde{\A}-\A\Vert_1+
\displaystyle\min_{\lambda\geq0}\Vert \b-\sum_{j\in J(0)}\lambda_j\nabla g_j(0)\Vert_1\\
\rho^\infty&=&\sup\left[\displaystyle\min_{\tilde{\A}\succeq0}\Vert \tilde{\A}-\A\Vert_\infty\:,\:
\displaystyle\min_{\lambda\geq0}\Vert \b-\sum_{j\in J(0)}\lambda_j\nabla g_j(0)\Vert_\infty\,\right].
\end{array}
\end{equation}
 
Observe that (\ref{inv-44}) can be solved in two steps. One first 
solves the problem $\min_{\lambda\geq0}\Vert \b-\sum_{j\in J(0)}\lambda_j\nabla g_j(0)\Vert_k$,
which is a linear program with finite value, hence with an optimal solution.
One next solves the problem $\min_{\tilde{\A}\succeq0}\Vert \tilde{\A}-\A\Vert_k$
which computes the $\ell_k$-projection of $\A$ onto the closed convex cone of positive semidefinite matrices
(a semidefinite program with an optimal solution).

\begin{lem}
\label{lem-convexity}
Let $\K\subset\R^n$ be as in (\ref{setk}) with $g_j$ being concave for every $j=1,\ldots,m$.
Then (\ref{inv-4}) has an optimal solution 
$\tilde{f}^*\in\R[\x]_{2}$ and $0$ is an optimal solution of the convex optimization problem
$\P':\min\{\tilde{f}^*(\x)\,:\,\x\in\K\}$.
\end{lem}
\begin{proof}
Let $(\tilde{f},\lambda)$ (with $\tilde{f}\in\R[\x]_{2}$) be any feasible solution of (\ref{inv-4}).
The constraint in (\ref{inv-4}) states that $\nabla L(0)=0$, where $L\in\R[\x]$
is the Lagrangian polynomial $\x\mapsto L(\x):=\tilde{f}(\x)-\sum_{j\in J(0)}\lambda_j\,g_j(\x)$, which 
is convex on $\K$ because the $g_j$'s are concave, the $\lambda_j$'s are nonnegative, and
$\tilde{f}$ is convex. Therefore $\nabla L(0)=0$ implies that
$0$ is a global minimizer of $L$ on $\R^n$ and a global minimizer of $\tilde{f}$ on $\K$ because
\begin{equation}
\label{lagrange}
\tilde{f}(\x)\,\geq\,L(\x)\,\geq\,L(0)\,=\,\tilde{f}(0),\qquad\forall\,\x\in\K.\end{equation}
It remains to prove that (\ref{inv-4}) has an optimal solution $\tilde{f}^*$. 
But we have seen that (\ref{inv-4}) is equivalent to (\ref{inv-44}) for which an optimal solution can be found
by solving a linear program and a semidefinite program.
\end{proof}

So in this case where the $g_j$'s are concave (hence $\K$ is convex), one obtains 
the convex programming problem with quadratic cost, whose criterion is the closest to $f$ for
the $\ell_k$-norm.

\section{Global $\epsilon$-optimality}
\label{section-epsilon}

One may be less demanding and ask $\y\in\K$ (or $0\in\K$ after a change of variable)
to be only a global $\epsilon$-minimizer. That is, one searches for a polynomial $\tilde{f}\in\R[\x]_{d_0}$ as close as possible to $f$ and such that
$\tilde{f}(\x)-\tilde{f}(0)\geq -\epsilon$ for all $\x\in\K$ and some $\epsilon>0$ fixed. Then we will see that one may approach
as closely as desired an optimal solution of the ideal inverse problem (\ref{inv-0}).
With $\K\subset [-1,1]^n$ as in (\ref{setk-1}), the analogue of Problem (\ref{inv-1}) reads:
\begin{equation}
\label{invep-1}
\begin{array}{rl}
\rho_{d\epsilon}^k:=\displaystyle\min_{\tilde{f},\sigma_j\in\R[\x] }&\Vert f-\tilde{f}\Vert_k\\
\mbox{s.t.}& \tilde{f}(\x)-\tilde{f}(0)+\epsilon\,=\,\displaystyle\sum_{j=0}^{m+n} \sigma_j(\x)\,g_j(\x),\quad\forall\x\in\R^n\\
&\tilde{f}\in\R[\x]_{d_0};\:\sigma_j\in\Sigma[\x]_{d-v_j},\quad j=0,1,\ldots,m+n.
\end{array}
\end{equation}
For instance, with $k=1$ (\ref{inv-3-k=1}) now reads
 \begin{equation}
\label{invep-3-k=1}
\begin{array}{rl}
\rho_{d\epsilon}^1:=\displaystyle\min_{\lambda\geq0,\tilde{f},\Z_j}&\displaystyle\sum_{\alpha\in\N^n_{d_0}\setminus\{0\}}\lambda_\alpha\\
\mbox{s.t.}& \lambda_\alpha +\tilde{f}_\alpha\geq f_\alpha,\quad\forall\alpha\in\N^n_{d_0}\setminus\{0\}\\
& \lambda_\alpha -\tilde{f}_\alpha\geq -f_\alpha,\quad\forall\alpha\in\N^n_{d_0}\setminus\{0\}\\
&\la\Z_0,\B_\alpha\ra+\displaystyle\sum_{j=1}^{m+n} \la\Z_j,\C^j_\alpha\ra\,=\,\left\{\begin{array}{l}
\tilde{f}_\alpha,\:\mbox{ if }0<\vert\alpha\vert\leq\,d_0\\
0,\:\mbox{ if $\vert\alpha\vert>d_0$}\\
\epsilon,\:\mbox{ if $\alpha=0$}\end{array}\right.\\
&\Z_j\succeq0,\quad j=0,1,\ldots,m+n.\end{array}
\end{equation}
while its dual reads
\begin{equation}
\label{epdual-k=1}
\left\{\begin{array}{ll}
(\rho^1_{d\epsilon})^*=\displaystyle\max_{\u,\v\geq0,\z}&-\epsilon z_0+\displaystyle
\sum_{\alpha\in\N^n_{d_0}\setminus\{0\}}f_\alpha(u_\alpha-v_\alpha)\:(=L_\z(f(0)-f-\epsilon))\\
\mbox{s.t.}&u_\alpha+v_\alpha\,\leq \,1,\quad\forall\,\alpha\in\N^n_{d_0}\setminus\{0\}\\
&u_\alpha-v_\alpha+z_\alpha \,=\,0,\quad\forall\alpha\in\N^n_{d_0}\setminus\{0\}\\
&\M_d(\z),\,\M_{d-v_j}(g_j\,\z)\,\succeq\,0,\quad j=1,\ldots,m+n.
\end{array}\right.
\end{equation}
Again with no loss of generality and possibly after a change of variable, we assume that $\y=0\in\K$.
\begin{lem}
\label{lem-ep}
Let $\K$ be as in (\ref{setk-1}) and let $\rho^k$ be the optimal value of the ideal inverse problem $\mathcal{P}$ in (\ref{inv-0}).
Then for every fixed $\epsilon>0$ there exists $d_\epsilon\in\N$ such that $\rho^k_{d\epsilon}\leq \rho^k$ for 
all $d\geq d_\epsilon$.
\end{lem}
\begin{proof}
Let $\tilde{f}^*\in\R[\x]_{d_0}$ be an optimal solution of the ideal inverse problem $\mathcal{P}$ with value $\rho^k=\Vert f-\tilde{f}^*\Vert_k$. 
Observe that the polynomial
$\x\mapsto \tilde{f}^*(\x)-\tilde{f}^*(0)+\epsilon$ is strictly positive on $\K$ and so by Theorem \ref{th-put}
it belongs to $Q(g)$; and so it belongs to $Q_d(g)$ as soon as $d\geq d_\epsilon$ (for some $d_\epsilon\in\N$).
Hence $\tilde{f}^*$ is a feasible solution of (\ref{invep-1}) which implies the desired result $\rho^k\geq\rho^k_{d\epsilon}$ for all $d\geq d_\epsilon$.
\end{proof}
The following analogue of Theorem \ref{th1} 
shows that the optimal value $\rho_{d\epsilon}^k$ of (\ref{invep-1}) 
is still helpful to bound the quantity $f(0)-f^*$
(where $f^*$ is the global optimum of problem $\P$). 

\begin{thm}
\label{th-ep}
Assume that $\K$ in (\ref{setk-1}) has nonempty interior and
let $\x^*\in\K$ be  a global minimizer of $\P$ with optimal value $f^*$.
For every $\epsilon>0$ fixed, let $\tilde{f}_{d\epsilon}\in\R[\x]_{d_0}$ be an optimal solution of (\ref{invep-1}) with optimal value $\rho^k_{d\epsilon}$. Then:

{\rm (a)}  $0\in\K$ is a global $\epsilon$-minimizer of the problem $\tilde{f}^*_{d\epsilon}=\min_\x\{\tilde{f}_{d\epsilon}(\x)\,:\,\x\in\K\}$. In particular,
if $\rho_{d\epsilon}^k=0$ then $\tilde{f}_{d\epsilon}=f$ and $0$ is a global $\epsilon$-minimizer of $\P$.\\

{\rm (b)} If $k=1$ then $f^*\,\leq\,f(0)\,\leq\,f^*+\epsilon+\rho_{d\epsilon}^1$.\\

{\rm (c)} If $k=\infty$ then $f^*\,\leq\,f(0)\,\leq\,f^*+\epsilon+s(d_0)\,\rho_{d\epsilon}^\infty$.
\end{thm}
The proof is omitted as being a verbatim copy of that of Theorem \ref{th1} and using
\[f^*=f(\x^*)=\underbrace{f(\x^*)-\tilde{f}_d(\x^*)}+\underbrace{\tilde{f}_d(\x^*)-\tilde{f}_d(0)}_{\geq -\epsilon}+\tilde{f}_d(0).\]
Concerning the canonical form associated with the $\ell_1$-norm we also have the following analogue of
Theorem \ref{th1-ell1}.
\begin{thm}
\label{th1ep-ell1}
Assume that $\K$ in (\ref{setk-1}) has a nonempty interior and for every $\epsilon>0$,
let $\tilde{f}_{d\epsilon}\in\R[\x]_{d_0}$ be an optimal solution of (\ref{invep-1}) with optimal value $\rho_{d\epsilon}^1$
for the $\ell_1$-norm. Then $\tilde{f}_{d\epsilon}$ is of the form:
\begin{equation}
\label{th1ep-ell1-1}
\tilde{f}_{d\epsilon}(\x)\,=\,f(\x)+\b'\x+\sum_{i=1}^n\lambda_i^*\,x_i^2,
\end{equation}
for some vector $\b\in\R^n$ and some 
nonnegative vectors $\lambda^*\in\R^n$, optimal solution of the semidefinite program:
\[\begin{array}{rl}
\rho_{d\epsilon}^1:=\displaystyle\min_{\lambda,\gamma,\b}&\displaystyle
\Vert\b\Vert_1+\sum_{i=1}^n\lambda_i\\
\mbox{s.t.}& f-f(0)+\b'\x+\displaystyle\sum_{i=1^n}\lambda_i\,x_i^2+\epsilon\in Q_{d}(g),\quad 
\lambda\,\geq0.
\end{array}\]
(And $\lambda^*=0$ if $d_0=1$.)
\end{thm}
We end up with the analogue of Proposition \ref{prop-asymptotics} for the asymptotics as $d\to\infty$.
For every $\epsilon>0$, let us call $\mathcal{P}_\epsilon$ the analogue of problem $\mathcal{P} \,(=\mathcal{P}_0)$, i.e.,
\begin{equation}
\label{invep-0}
\mathcal{P}_\epsilon:\quad\rho^k_\epsilon=
\displaystyle\min_{\tilde{f}\in\R[\x]_{d_0}}\:\{\,\Vert \tilde{f}-f\Vert_k\::\:\x\mapsto \tilde{f}(\x)-\tilde{f}(\y)+\epsilon\,\in\,\psd_{d_0}(\K)\:\}.
\end{equation}
\begin{prop}
\label{thep-ideal}
For every $\epsilon>0$ fixed, Problem (\ref{invep-0}) has an optimal solution $\tilde{f}^*_\epsilon\in\R[\x]_{d_0}$,
and $\rho^k_\epsilon=0$ if and only if $\y$ is a global $\epsilon$-minimizer of $\P$. 
\end{prop}
The proof is similar to that of Theorem \ref{inv-0}.
Next, interestingly, we are able to relate (\ref{invep-1}) and the ideal inverse problem (\ref{inv-0}) as $\epsilon\to 0$.

\begin{prop}
\label{thep-ep=0}
Let $\K$ in (\ref{setk-1}) be with nonempty interior.
Let $\y=0\in\K$ and $\rho^k$ be the optimal value of the ideal inverse problem $\mathcal{P}$ in (\ref{inv-0})
and let $\tilde{f}_{d\epsilon}\in\R[\x]_{d_0}$ (resp. $\tilde{f}^*_\epsilon\in\R[\x]_{d_0}$) be any optimal solution of
(\ref{invep-1}) (resp. (\ref{invep-0})). 

(i) Let $\epsilon_\ell>0$, $\ell\in\N$, be such that $\epsilon_\ell\to0$ as $\ell\to\infty$.
 Then
every accumulation point $\hat{f}\in\R[\x]_{d_0}$ of the sequence $(\tilde{f}^*_{\epsilon_\ell})\subset\R[\x]_{d_0}$, $\ell\in\N$, 
is an optimal solution of the ideal inverse problem (\ref{inv-0}).

(ii) If for every $\epsilon_\ell>0$,  $d_\ell\in\N$ is sufficiently large, every accumulation point of the sequence
$(\tilde{f}_{d_\ell\epsilon_\ell})\subset\R[\x]_{d_0}$ is an optimal solution of the ideal inverse problem (\ref{inv-0}).
\end{prop}
\begin{proof}
(i) As $\Vert \tilde{f}^*_{\epsilon_\ell}-f\Vert_k\leq \rho^k$ for all $\ell$, the sequence
$(\tilde{f}^*_{\epsilon_\ell})\subset\R[\x]_{d_0}$, $\ell\in\N$, has accumulation points. So consider an arbitrary 
converging subsequence (still denoted $(\tilde{f}^*_{\epsilon_\ell})$ for notational convenience) $\tilde{f}^*_{\epsilon_\ell}\to\hat{f}\in\R[\x]_{d_0}$, as $\ell\to\infty$.

Observe that $\tilde{f}^*_{\epsilon_\ell}(\x)-\tilde{f}^*_{\epsilon_\ell}(0)+\epsilon_\ell\geq 0$ for all $\x\in\K$ and all $\ell\in\N$.
With $\x\in\K$ fixed, arbitrary, letting $\ell\to\infty$ yields $\hat{f}(\x)-\hat{f}(0)\geq 0$, and so 
$\x\mapsto \hat{f}(\x)-\hat{f}(0)\in\psd_{d_0}(\K)$. Moreover, 
$\rho^k_{\epsilon_\ell}\leq\rho^k$ for all $\ell$ yields $\Vert f-\hat{f}\Vert_k\leq \rho^k$, which in turn implies 
that $\hat{f}$ is an optimal solution of the inverse problem (\ref{inv-0}). 
As the accumulation point $\hat{f}$ was arbitrary the result follows.
The proof of (ii) is similar 
if one recalls that by Lemma \ref{lem-ep}, $\rho^k_{d\epsilon_\ell}=\Vert f-\tilde{f}_{d\epsilon_\ell}\Vert_k\leq\rho^k$ for all $\ell$ and 
all $d$ sufficiently large, say $d\geq d_{\ell}$.
\end{proof}

Hence, by asking $\y\,(=0)\in\K$ to be only a global $\epsilon$-minimizer, one may obtain 
a polynomial $\tilde{f}_{d\epsilon}\in\R[\x]_{d_0}$ as close as desired to an optimal solution of the ideal inverse problem
provided that $\epsilon$ is sufficiently small and $d$ is sufficiently large.

\subsection{Illustrative examples and discussion}
We here provide some simple illustrative examples and
show that the representation of the set $\K$ may be important for getting a Putinar certificate 
faster.
\begin{ex}
\label{ex1}
{\rm Let $n=2$ and consider the optimization problem
$\P:\:f^*\,=\,\min_\x\:\{f(\x)\,:\,\x\in\K\}$
with $\x\mapsto f(\x)=x_1+x_2$, and
\[\K\,=\,\{\x\in\R^2\::\: x_1x_2\geq1;\:  1/2\leq \x\leq 2\:\}.\]
The polynomial $f$ is convex and the set $\K$ is convex as well, but the polynomials that define $\K$ are not all concave.
That is, $\P$ is a convex optimization problem, but not a convex programming problem. 
The point $\y=(1,1)\in\K$ is a global minimizer and the KKT conditions at $\y$ are satisfied with
$\lambda=(1,0,0)\in\R^3$, i.e.,
\[\nabla f(\x)-\lambda_1\nabla g_1(\x)\,=\,0\mbox{ with }\x=(1,1)\mbox{ and }\lambda_1 =1.\]
However, the Lagrangian
\[\x\mapsto L(\x)\,:=\,f(\x)-f^*-\lambda_1\,g_1(\x)\,=\,x_1+x_2-1-x_1x_2,\]
is not convex and $(1,1)$ is not a global minimizer of $L$ on $\R^2$.
This example just illustrates the fact that even in the convex case 
where the $g_j$'s are not concave, the KKT conditions do not provide a {\it certificate} of global optimality, contrary to ``convex programming" where since $L$ is now convex, obviously using $L(\x)\geq L(\y)=0$ (because $\nabla L(\y)=0$),
\[f(\x)-f^*\geq\,L(\x)\,\geq\,L(\y)\,=\,0,\]
whenever $\x\in\K$, and so $f(\x)\geq f^*$ for all $\x\in\K$, the desired certificate of global optimality. 

Next, if we now use the test of inverse optimality with $d=1$, 
one searches for a polynomial $\tilde{f}_d$ of degree at most $d_0=1$, and such that
\[\tilde{f}_d(\x)-\tilde{f}_d(1,1)\,=\,\sigma_0(\x)+\sigma_1(\x)(x_1x_2-1)+\sum_{i=1}^2\psi_i(\x) (2-x_i)
+\phi_i(\x) (x_i-1/2),\]
for some s.o.s. polynomials $\sigma_1,\psi_i\,\phi_i\in\Sigma[\x]_0$ and some s.o.s. polynomial $\sigma_0\in\Sigma[\x]_1$.
But then necessarily $\sigma_1=0$ and $\psi_i,\phi_i$ are constant, which in turn implies that
$\sigma_0$ is a constant polynomial. A straightforward calculation shows that $\tilde{f}_1(\x)=0$ for all $\x$, and so 
$\rho^1_1=2$. And indeed this is confirmed when solving\footnote{To solve
(\ref{dual-k=1}) we have used the GloptiPoly software
of Henrion et al. \cite{gloptipoly}, and dedicated to solving the Generalized Problem of Moments whose 
problem (\ref{dual-k=1}) is only a special case.} (\ref{dual-k=1}) with $d=1$. Solving again (\ref{dual-k=1}) with now $d=2$ yields
$\rho^1_2=2$ (no improvement) and with $d=3$ we obtain the desired result $\rho^1_3=0$.

On the other hand, if now $\K$ has the representation:
\[\{\x\::\:x_1x_2-1\,\geq\,0;\quad (x_i-1/2)(2-x_i)\geq\,0;\quad i=1,2\:\},\]
then the situation differs because in fact 
\[x_1+x_2-2\,=\,\frac{1}{5}+\frac{2}{5}(x_1-x_2)^2+\frac{4}{5}(x_1x_2-1)+\frac{2}{5}\sum_{i=1}^2(x_i-1/2)(2-x_i),\]
i.e., $f-f^*$ has a Putinar's certificate with degree bound $d=1$. Hence the test of inverse optimality yields $\rho^1_1=0$ with $\tilde{f}_1=f$.
}\end{ex}
The above example illustrates that the representation of $\K$ may be important.
\begin{ex}
\label{newex1}
{\rm Again consider  Example \ref{ex1} but now with
$\y=(1.1,1/1.1)\in\K$, which is not a global optimum of $f$ on $\K$ any more. By solving
(\ref{dual-k=1}) with $d=1$ we still find $\rho^1_1=2$ (i.e., $\tilde{f}_1=0$), and with $d=2$ we find $\rho_2\approx 0.1734$ and
$\tilde{f}_2(\x)\approx 0.8266\,x_1+x_2$.
And indeed by solving (using GloptiPoly) the new optimization problem with criterion $\tilde{f}_2$ 
we find the global minimizer $(1.1,0.9091)\approx\y$. With $d=3$ we obtains the same value $\rho^1_3=0.1734$, suggesting (but with no guarantee)
that $\tilde{f}_2$ is already an optimal solution of the ideal inverse problem.
}\end{ex}
\begin{ex}
\label{newex2}
{\rm Consider  now the disconnected set $\K:=\{\x\,:\, x_1x_2\geq1;\:x_1^2+x_2^2 <= 3\}$
and the non convex criterion $\x\mapsto f(\x):=-x_1-x_2^2$ for which 
$\x^*=(-0.618,-1/0.618)\in\partial\K$ is the unique global minimizer.
Let $\y:=(-0.63,-1/0.63)\in\partial\K$ for which the constraint $x_1x_2\geq1$ is active.
At steps $d=2$ and $d=3$ one finds that
 $\tilde{f_d}=0$ and $\rho^1_d=\Vert f\Vert_1$. That is, $\y$ is a global minimizer of the trivial
 criterion$\tilde{f}(\x)=0$ for all $\x$, and cannot be a global minimizer
 of some non trivial polynomial criterion. 
 
 Now let  $\y=(-0.63,-\sqrt{3-0.63^2})$ so that the constraint $x_1^2+x_2^2<=3$ is active.
 With obtain $\rho^1_1=\Vert f\Vert_1$ and $\tilde{f}_1=0$. With $d=2$
 we obtain $\tilde{f}_2=1.26\,x_1-x_2^2$. With $d=3$ we obtain the same result, suggesting (but with no guarantee) that
 $\tilde{f}_2$ is already an optimal solution of the ideal inverse optimization problem.
 }\end{ex}

\begin{ex}
{\rm Consider the MAXCUT problem $\max \{\x'\A\x\,:\,\x_i^2=1,i=1,\ldots,n\}$ where 
$\A=\A'\in\R^{n\times n}$ and $\A_{ij}=1/2$ for all $i\neq j$. For $n$ odd,
an optimal solution is $\y=(y_j)$ with $y_j=1$, $j=1,\ldots\lceil n/2\rceil$, and $y_j=-1$ otherwise.
However, the first semidefinite relaxation 
\[\displaystyle\max\:\{\lambda \::\: \x'\A\x-\lambda=\sigma+\sum_{j=1}^n\gamma_i(x_i^2-1)
;\:\sigma\in\Sigma[\x]_1;\:\lambda,\gamma\in\R\}\]
provides the lower bound $-n/2$ (with famous Goemans-Williamson ratio guarantee).
So $\y$ cannot be obtained from the first semidefinite relaxation even though it is
an optimal solution. The inverse optimization problem reads:
Find the quadratic form $\x\mapsto\x'\tilde{\A}\x$ such that 
$\displaystyle\x'\tilde{\A}\x-\y'\tilde{\A}\y\,=\,\sigma+\sum_{j=1}^n\gamma_i(x_i^2-1)$,
for some $\sigma\in\Sigma[\x]_1,\:\lambda,\gamma\in\R$, and 
which minimizes the $\ell_1$-norm $\Vert\A-\tilde{\A}\Vert_1$. This is an inverse optimization problem with structural constraints 
as described in Section \ref{structural} (since we search for a quadratic form
and not an arbitrary quadratic polynomial $\tilde{f}_2$). Hence, solving (\ref{inv-33-k=1}) for $n=5$ with $\y$ as above, we find that
\[\tilde{\A}\,=\,\frac{1}{2}\left[\begin{array}{ccccc} 0&2/3&2/3&1&1\\2/3&0&2/3&1&1\\
2/3&2/3&0&1&1\\
1&1&1&0&1\\
1&1&1&1&0\end{array}\right],\]
that is, only the entries $(i,j)\in\{(1,2),(1,3),(2,3)\}$ are modified from $1/2$ to $1/3$. 
}\end{ex}

\section*{Conclusion}
We have presented a paradigm for inverse polynomial optimization.
Crucial is Putinar's Positivstellensatz which provides us
with the desired certificate of global optimality for a given
feasible point $\y\in\K$ and a candidate criterion $\tilde{f}$. In addition, to some extent,
the size of the certificate can be adapted to the computational capabilities available. Finally, and remarkably, when using the
$\ell_1$-norm the resulting inverse optimal criterion $\tilde{f}$ has a simple and explicit
{\it canonical} form. We hope that the concept of inverse optimization will receive more attention
from the optimization community as it could even provide an alternative stopping criterion
at the current iterate $\y\in\K$ of any local optimization algorithm for solving the original problem $\P$.

\end{document}